\documentclass{amsart}
\usepackage[a4paper, margin=3cm]{geometry}
\usepackage{amsthm}
\usepackage{amsmath,amssymb}
\usepackage{mathrsfs}
\usepackage{url}
\newtheorem{theorem}{Theorem}[section]
\newtheorem*{theorem*}{Theorem}
\newtheorem{lemma}[theorem]{Lemma}
\newtheorem*{lemma*}{Lemma}
\newtheorem{proposition}[theorem]{Proposition}
\newtheorem*{proposition*}{Proposition}
\newtheorem{corollary}[theorem]{Corollary}
\newtheorem*{corollary*}{Corollary}

\newtheorem*{claim*}{Claim}

\newtheorem*{fact*}{Fact}
\newtheorem{conjecture}[theorem]{Conjecture}
\newtheorem*{conjecture*}{Conjecture}
\theoremstyle{definition}
\newtheorem{definition}[theorem]{Definition}
\newtheorem*{definition*}{Definition}

\newtheorem*{example*}{Example}
\newtheorem{remark}[theorem]{Remark}
\newtheorem*{remark*}{Remark}

\newtheorem*{question*}{Question}

\newcommand{\abs}[1]{\left\lvert#1\right\rvert}
\DeclareMathOperator{\Irr}{Irr}
\DeclareMathOperator{\End}{End}
\DeclareMathOperator{\rank}{rank}
\DeclareMathOperator{\toral}{toral}
\DeclareMathOperator{\Frob}{Frob}
\DeclareMathOperator{\Root}{root}
\DeclareMathOperator{\Ad}{Ad}
\DeclareMathOperator{\Aa}{a}
\DeclareMathOperator{\ind}{ind}
\DeclareMathOperator{\vol}{vol}
\DeclareMathOperator{\Hom}{Hom}
\address{graduate school of mathematical science, the university of tokyo, 3-8-1 komaba, meguroku, tokyo 153-8914, japan.}
\email{kohara@ms.u-tokyo.ac.jp}
\subjclass[2020]{22E50}
\begin{document}
\title{on the formal degree conjecture for non-singular supercuspidal representations}
\author{Kazuma Ohara}
\date{}
\maketitle
\begin{abstract}
We prove the formal degree conjecture for non-singular supercuspidal representations based on Schwein's work \cite{2021arXiv210100658S} proving the formal degree conjecture for regular supercuspidal representations. The main difference between our work and Schwein's work is that in non-singular case, the Deligne--Lusztig representations can be reducible, and the $S$-groups are not necessary abelian. Therefore, we have to compare the dimensions of irreducible constituents of the Deligne--Lusztig representations and the dimensions of irreducible representations of $S$-groups.
\end{abstract}
\section{Introduction}
Let $F$ be a non-archimedean local field of residue characteristic $p$, and let $G$ be a connected reductive group over $F$. For an essentially discrete series representation $\pi$ of $G(F)$, we can consider a positive constant $d(\pi)$ called the \emph{formal degree} of $\pi$.
In some sense, we can consider the formal degree as a generalization of the dimension of a finite-dimensional representation.

Recall that there is a conjectural correspondence called the \emph{local Langlands correspondence} between the set of irreducible representations of $G(F)$, and a set of pairs $(\varphi, \rho)$, where $\varphi\colon W_F \times \text{SL}_2(\mathbb{C}) \to {^{L}G}$ is an $L$-parameter, and $\rho$ is an irreducible representation of the $S$-group $\pi_0(S_{\varphi}^{\text{sc}})$, which is determined by $\varphi$.
In \cite{MR2350057}, Hiraga, Ichino, and Ikeda proposed a conjecture called the \emph{formal degree conjecture}, which predicts that the formal degree $d(\pi)$ can be described by using the pair $(\varphi, \rho)$ attached to $\pi$. More precisely, they predicted that the formal degree $d(\pi)$ of $\pi$ is equal to
\[\frac{\text{dim}(\rho)}{\abs{\pi_0(S_{\varphi}^{\natural})}} \cdot \abs{\gamma(0, \pi, \text{Ad}, \psi)},\]
where $\gamma$ denotes the $\gamma$-factor defined in \cite[1]{MR2350057}, and $\pi_0(S_{\varphi}^{\natural})$ is a finite group determined by $\varphi$.

The formal degree conjecture depends itself on a conjectural correspondence, the local Langlands correspondence.
Hence, in order to verify the formal degree conjecture, we must first have access to a candidate for the local Langlands correspondence. 
In \cite{MR4013740}, Kaletha defined and constructed a large class of supercuspidal representations which he calls \emph{regular}, and gave a candidate for the local Langlands correspondence for regular supercuspidal representations. 
Kaletha's construction starts with a pair $(S, \theta)$, where $S$ is a tame elliptic maximal torus, and $\theta$ is a character of $S(F)$ which satisfy some conditions. From the pair $(S, \theta)$, he constructed a generic cuspidal $G$-datum of \cite{MR1824988}, and then he obtained a regular supercuspidal representation by Yu's construction. 
In \cite{2021arXiv210100658S}, Schwein calculated the formal degrees of supercuspidal representations which are obtained by Yu's construction and proved that the formal degree conjecture is true for the local Langlands correspondence in \cite{MR4013740}.

On the other hand, in \cite{2019arXiv191203274K}, Kaletha defined and constructed a wider class of supercuspidal representations which he calls \emph{non-singular}, and also gave a candidate for the local Langlands correspondence for non-singular supercuspidal representations. In this paper, we prove that the formal degree conjecture also holds for the local Langlands correspondence in \cite{2019arXiv191203274K}.

Kaletha's construction of non-singular supercuspidal representations and the local Langlands correspondence for them are very similar to the one for regular supercuspidal representations.  Therefore, we can apply the arguments in \cite{2021arXiv210100658S} in our case. However, both in the automorphic side and the Galois side, there are some differences.
In the case of regular supercuspidal representations, the Deligne--Lusztig representations appearing in the construction are irreducible, so we can calculate the formal degrees of the regular supercuspidal representations by using the dimension formula for Deligne--Lusztig representations. Moreover, in this case, the $S$-groups are abelian, so the factor of the dimension $\text{dim}(\rho)$ of the irreducible representation of the $S$-group in the formal degree conjecture is equal to $1$.
In the non-singular case, however, the Deligne--Lusztig representations appearing in the construction are not necessary irreducible, and the $S$-groups are not necessary abelian. 
Kaletha constructed the bijections between the sets of irreducible constituents of the Deligne--Lusztig representations and some sets of irreducible representations of the $S$-groups, from which he constructed the local Langlands correspondence for non-singular supercuspidal representations. We compare the dimensions of an irreducible constituent of the Deligne--Lusztig representation and the corresponding representation of the $S$-group, and by using this comparison, we prove the formal degree conjecture for non-singular case.

We sketch the outline of this paper. In Section~\ref{secfdc}, we state the formal degree conjecture. In Section~\ref{secyu}, we explain Yu's construction of supercuspidal representations \cite{MR1824988} and the calculation of the formal degrees of these representations \cite{2021arXiv210100658S}. In Section~\ref{secnonsing}, we review the definition and the construction of non-singular supercuspidal representations \cite{2019arXiv191203274K}. Here, we explain the description of the sets of irreducible constituents of the Deligne--Lusztig representations in \cite{2019arXiv191203274K}. In Section~\ref{seclpara}, we introduce the notions of torally wild $L$-parameters and torally wild $L$-packet data, which are defined in \cite{2019arXiv191203274K}, and explain the way to construct $L$-packets from these data. In Section~\ref{secend}, we explain the parametrization of the elements in the $L$-packet by irreducible representations of the $S$-group. Section~\ref{seclem} is our essential part.  In this section, we prove a variant of \cite[Proposition~B.3]{2019arXiv191203274K}. Using the results of this section, we calculate the formal degrees of non-singular supercuspidal representations in Section~\ref{secfdns}. Finally, in Section~\ref{seccomp}, we prove that the formal degree conjecture holds for the local Langlands correspondence for non-singular supercuspidal representations.
\subsection*{Acknowledgment}
I am deeply grateful to my supervisor Noriyuki Abe for his enormous support and helpful advice. He checked the draft and gave me useful comments.
I would also like to thank Tasho Kaletha. He answered my question about twisted Yu's construction and sent me a preliminary copy of their paper about twisted Yu's construction.
I am supported by the FMSP program at Graduate School of Mathematical Sciences, the University of Tokyo.
\section{Notation and assumptions}
Let $F$ be a non-archimedean local field of residue characteristic $p$, and let $k_F$ be its residue field.
We fix an algebraic closure $\overline{F}$ of $F$ and write $F^{\text{sep}}$ for the separable closure of $F$ in $\overline{F}$. We denote by $\overline{k_F}$ be the residue field of $F^{\text{sep}}$.
We write $F^{u}$ for the maximal unramified extension of $F$. For a field extension $E$ over $F$, we write $\mathcal{O}_{E}$ for its ring of integers.

We denote by $\Gamma_F$ the absolute Galois group $\text{Gal}(F^{\text{sep}}/F)$ of $F$, and by $W_F$ the Weil group of $F$.
We also denote by $I_F$ the inertia subgroup of $\Gamma_F$, by $P_F$ the wild inertia subgroup of $\Gamma_F$, and by $\text{Frob}$ the geometric Frobenius element in $\Gamma_F/I_F$.
For $r\ge 0$, let $W_{F}^{r}$ be the $r$-th subgroup of $W_F$ in its upper numbering filtration computed with respect to the unique discrete valuation $\text{ord}_{F}$ on $F^{\times}$ with the value group $\mathbb{Z}$.
For a representation $\pi$ of $W_F$, the depth of $\pi$ is defined as
\[\text{depth}(\pi)= \inf\{r\in \mathbb{R}_{\ge 0} \mid \pi(W_{F}^{r})=1\}.\]
For a field extension $E$ over $F$, we compute the upper numbering filtration $\{W_{E}^{r}\}_{r\in \mathbb{R}_{\ge 0}}$ of the Weil group $W_E$ of $E$ by using the unique extension of $\text{ord}_{F}$ to $E^{\times}$.  We compute the depth of a representation of $W_E$ by using this filtration.
We also define the filtration $\{E_{r}^{\times}\}_{r\in \mathbb{R}_{>0}}$ of $E^{\times}$ as
\[E_{r}^{\times}=\{x\in E^{\times}\mid \text{ord}_{F}(x-1)\ge r\}.\]
Here, we write the unique extension of $\text{ord}_{F}$ to $E^{\times}$ by the same symbol.
We let $E_{r+}^{\times}=\cup_{s>r} E_{s}^{\times}$ for $r \ge 0$.

Let $G$ be a connected reductive group over $F$ that splits over a tamely ramified extension of $F$.
We denote by $G_{\text{der}}$ the derived subgroup of $G$, by $Z(G)$ the center of $G$, by $Z$ the center of $G_{\text{der}}$, and by $A$ the maximal split central torus of $G$. 
Let $G^{\text{a}}$ denote the reductive group $G/A$. For a maximal torus $T$ of $G$, we also define $T^{\text{a}}$ be the torus $T/A$.

For a connected reductive group $H$ defined over $F$ or $k_F$, we denote by $\text{dim}(H)$ the dimension of $H$ and denote by $\text{rank}(H)$ the absolute rank of $H$. 
For a maximal torus $T$ of $H$, let $X^{*}(T)$ be the character group of $T$, $X_{*}(T)$ be the cocharacter group of $T$, and $R(H, T)$ be the absolute root system of $H$ with respect to $T$.
For a root $\alpha \in R(H, T)$, we denote by $\check{\alpha}$ the corresponding coroot.
When $H$ is defined over $F$, we denote by $\widehat{H}$ and $^{L}H$ the dual group of $H$ and the $L$-group of $H$ respectively.

We assume that $p$ is an odd prime, is not a bad prime for any irreducible factor of the absolute root system of $G$ in the sense of \cite[4.1]{springer1970conjugacy}, dose not divide the number of connected components of $Z(G)$, and dose not divide the order of the fundamental group of $G_{\text{der}}$.
Let $q$ be the cardinality of $k_{F}$, and let $\text{exp}_q \colon \mathbb{R} \to \mathbb{R}$ be the function $t \mapsto q^{t}$.

For a connected reductive group $H$ over $F$, we denote by $\mathcal{B}(H, F)$ the enlarged Bruhat--Tits building of $H$ over $F$ and denote by $\mathcal{B}^{\text{red}}(H, F)$ the reduced building of $H$ over $F$.
If $T$ is a maximal, maximally split torus of $H_E := H \times_{F} E$ for a field extension $E$ over $F$, then $\mathcal{A}(T, E)$ and $\mathcal{A}^{\text{red}}(T, E)$ denote the apartment of $T$ inside the Bruhat--Tits building $\mathcal{B}(H_E, E)$ and reduced building $\mathcal{B}^{\text{red}}(H_E, E)$ of $H_E$ over $E$ respectively.
For any $y\in \mathcal{B}(H,F)$, we denote by $[y]$ the projection of $y$ on the reduced building and by $H(F)_y$ (resp.\,$H(F)_{[y]}$) the subgroup of $H(F)$ fixing $y$ (resp.\,$[y]$).
For $y\in \mathcal{B}(H, F)$ and $r\in \widetilde{\mathbb{R}}_{\ge 0}=\mathbb{R}_{\ge 0}\cup \{r+\mid r\in \mathbb{R}_{\ge 0}\}$, we write $H(F)_{y, r}$ for the Moy--Prasad filtration subgroup of $H(F)$ of depth $r$ \cite{MR1253198, MR1371680}.
In the case that $H$ is a torus, we omit the notion $y$ and write $H(F)_{r}$ for the Moy-Prasad filtration subgroup of $H(F)$ of depth $r$, which dose not depend on $y$. 
For an irreducible representation $\pi$ of $H(F)$, we denote by $\text{depth}(\pi)$ the depth of $\pi$ in the sense of \cite[Theorem~3.5]{MR1371680}.

For a group $H$ and an $H$-module $M$, we denote by $M^{H}$ and $M_{H}$ its invariant and coinvariant respectively.
If $H$ is a cyclic group generated by $\sigma \in H$, let $M^{\sigma}$ and $M_{\sigma}$ denote $M^{H}$ and $M_{H}$ respectively.
  
For a finite-dimensional representation $\rho$ of a group $H$, we denote by $\text{dim}(\rho)$ the dimension of $\rho$.
Suppose that $K$ is a subgroup of $H$ and $h\in H$. We denote $hKh^{-1}$ by $^hK$. If $\rho$ is a representation of $K$, let $^h\!\rho$ denote the representation $x\mapsto \rho(h^{-1}xh)$ of $^hK$. 

Throughout the paper, we fix a prime number $\ell \neq p$ and an isomorphism $\mathbb{C} \simeq \overline{\mathbb{Q}_{\ell}}$.
\section{The formal degree conjecture}
\label{secfdc}
In this section, we explain the formal degree conjecture \cite[Conjecture~1.4]{MR2350057} of Hiraga, Ichino, and Ikeda. 

Let $(\pi, V)$ be an essentially discrete series representation of $G(F)$, i.\,e., it becomes a discrete series representation after twisting by a character of $G(F)$. We note that all supercuspidal representations of $G(F)$ are essentially discrete series representations.
When $\pi$ is a discrete series representation, the formal degree of $\pi$ associated with a Haar measure $\mu$ on $G(F)/A(F)$ is defined by the positive constant $\text{deg}(\pi, \mu)$ which satisfies 
\[\int_{G(F)/A(F)} (\pi(g)u, u')\overline{(\pi(g)v, v')} \ d\mu(g)= \text{deg}(\pi, \mu)^{-1} (u, v)\overline{(u', v')}\]
for all $u, u', v, v'\in V$, where $(\cdot, \cdot)$ denotes an invariant hermitian inner product on $V$.
In general, we define the formal degree of $\pi$ as the formal degree of any discrete series representation of $G(F)$ obtained from $\pi$ by twisting by a character of $G(F)$.
We fix a level-zero additive character $\psi$ of $F$, and let $\mu$ be the Haar measure on $G(F)/A(F)$ attached to $\psi$ as in \cite{MR1458303}. We write $d(\pi)=\text{deg}(\pi, \mu)$.
\begin{remark}
In \cite{MR2350057}, Hiraga, Ichino, and Ikeda use a different Haar measure to state the formal degree conjecture. However, the conjecture is modified in \cite{MR2425185}, in which the Haar measure above is used.
\end{remark}
Let $\varphi\colon W_F \times \text{SL}_2(\mathbb{C}) \to {^{L}G}$ be a discrete $L$-parameter.
We write $S_{\varphi}$ for the centralizer of $\varphi(W_F \times \text{SL}_2(\mathbb{C}))$ in $\widehat{G}$. 
Let $(\widehat{G})_{\text{der}}$ be the derived subgroup of $\widehat{G}$, $(\widehat{G})_{\text{ad}}$ be its adjoint group, and $(\widehat{G})_{\text{sc}}$ be the simply connected cover of $(\widehat{G})_{\text{der}}$.
Let $S_{\varphi}^{\text{ad}}$ be the image of $S_{\varphi}$ in $(\widehat{G})_{\text{ad}}$ and $S_{\varphi}^{\text{sc}}$ be the preimage of $S_{\varphi}^{\text{ad}}$ in $(\widehat{G})_{\text{sc}}$.
We also define $S_{\varphi}^{\natural}$ be the preimage of $S_{\varphi}$ via the natural map $\widehat{G^{\text{a}}} \to \widehat{G}$.
For these groups, let $\pi_0(*)$ denote the groups of connected components. 

It is believed that the $L$-parameter $\varphi$ determines a finite set $\prod_{\varphi}(G)$ of irreducible representations of $G(F)$ called $L$-packet, whose elements are indexed by a set of irreducible representations of $\pi_0(S_{\varphi}^{\text{sc}})$ \cite{MR2217572}.
We now state the formal degree conjecture, which depends on the conjectural correspondence above.
\begin{conjecture}[{\cite[Conjecture~1.4]{MR2350057}}]
\label{fdc}
Let $\varphi$ be a discrete $L$-parameter which corresponds to the $L$-packet $\prod_{\varphi}(G)$. Assume that $\pi \in \prod_{\varphi}(G)$ corresponds to the representation $\rho$ of $\pi_0(S_{\varphi}^{\text{sc}})$. Then
\[d(\pi)= \frac{\dim(\rho)}{\abs{\pi_0(S_{\varphi}^{\natural})}} \cdot \abs{\gamma(0, \pi, \Ad, \psi)},\]
where $\gamma$ is the $\gamma$-factor defined in \cite[1]{MR2350057}.
\end{conjecture}
\section{Yu's construction and formal degree}
\label{secyu}
In this section, we review Yu's construction of supercuspidal representations in \cite{MR1824988} and the calculations of the formal degrees of these representations, which are done in \cite{2021arXiv210100658S}.

An input for Yu's construction of supercuspidal representations of $G(F)$ is a tuple $\Psi=(\overrightarrow{G}, y, \overrightarrow{r}, \rho_{-1}, \overrightarrow{\phi})$ called a generic cuspidal $G$-datum, where
\begin{description}
\item[\bf{D1}]
$\overrightarrow{G}=\left(G^0\subsetneq G^1 \subsetneq\ldots \subsetneq G^d=G\right)$ is a sequence of twisted Levi subgroups of $G$ that split over a tamely ramified extension of $F$, i.\,e., there exists a tamely ramified extension $E$ of $F$ such that $G^i_{E}$ is split for $0\le i\le d$, and $\left(G^0_{E}\subsetneq G^1_{E} \subsetneq\ldots \subsetneq G^d_{E}=G_{E}\right)$ is a split Levi sequence in $G_{E}$ in the sense of \cite[Section~1]{MR1824988}; we assume that $Z(G^0)/Z(G)$ is anisotropic;
\item[\bf{D2}]
$y$ is a point in $\mathcal{B}(G^0, F)\cap \mathcal{A}(T, E)$ whose projection on the reduced building of $G^0(F)$ is a vertex, where $T$ is a maximal torus of $G^0$ (hence of $G^i$) whose splitting field $E$ is a tamely ramified extension of $F$; 
\item[\bf{D3}]
$\overrightarrow{r}=\left(r_0, \ldots , r_d\right)$ is a sequence of real numbers satisfying
\begin{align*}
\begin{cases}
0<r_0<r_1<\cdots <r_{d-1}\le r_{d} & (d>0),\\
0\le r_0 & (d=0);
\end{cases}
\end{align*}
\item[\bf{D4}]
$\rho_{-1}$ is an irreducible representation of $G^0(F)_{[y]}$ such that $\rho_{-1}\restriction_{G^0(F)_{y, 0}}$ is the inflation of a cuspidal representation of $G^0(F)_{y,0}/ G^0(F)_{y, 0+}$;
\item[\bf{D5}]
$\overrightarrow{\phi}=\left(\phi_0, \ldots , \phi_d\right)$ is a sequence of characters, where $\phi_i$ is a character of $G^i(F)$; we assume that $\phi_i$ is trivial on $G^i(F)_{y, r_i+}$ but non-trivial on $G^i(F)_{y, r_i}$ for $0\le i\le d-1$. If $r_{d-1}<r_d$, we assume that $\phi_d$ is trivial on $G^d(F)_{y, r_d+}$ but non-trivial on $G^d(F)_{y, r_d}$, otherwise we assume that $\phi_d=1$. Moreover, we assume that $\phi_i$ is $G^{i+1}$-generic of depth $r_i$ relative to $y$ in the sense of \cite[Section.~9]{MR1824988} for $0\le i\le d-1$. 
\end{description}
From this datum, Yu constructed a pair $(K^{d}, \rho_{d})$ consisting of a compact-mod-center open subgroup and its irreducible representation, and obtained an irreducible supercuspidal representation $\pi_{\Phi}=\ind_{K^d} ^{G(F)} \rho_{d}$ of $G(F)$, where $\ind_{K^d}^{G(F)} \rho_{d}$ denotes the compactly induced representation \cite[Theorem~15.1]{MR1824988}. The formal degree of $\pi_{\Phi}$ is calculated in \cite{2021arXiv210100658S}.
Recall that $A$ denotes the maximal split central torus of $G$, and $G^{\text{a}}$ denotes the reductive group $G/A$. We also define $G^{\text{a}, i}=G^{i}/A$ for $0\le i \le d$.
Let $(\mathsf{G^{\text{a}, i}})_{[y]}^\circ$ be the reductive quotient of the special fiber of the connected parahoric group scheme of $G^{\text{a}, i}$ associated to $y$, which is a reductive group over $k_F$.
\begin{proposition}[{\cite[Theorem~A]{2021arXiv210100658S}}]
\label{thmA}
Let $\Phi$ be a generic cuspidal $G$-datum. Then, the formal degree $d(\pi_{\Phi})$ of $\pi_{\Phi}$ is equal to
\[\frac{\dim(\rho_{-1})}{[G^{\Aa, 0}(F)_{[y]} : G^{\Aa, 0}(F)_{y, 0+}]} \exp_{q} \left(\frac{1}{2} \dim(G^{\Aa}) + \frac{1}{2} \dim((\mathsf{G^{\Aa, 0}})_{[y]}^\circ) + \frac{1}{2}\sum_{i=0}^{d-1} r_i \left(\abs{R(G^{i+1}, T) - R(G^{i}, T)}\right)\right).\]  
\end{proposition}
\section{An review of non-singular supercuspidal representations}
\label{secnonsing}
In this section, we review the definition and the construction of non-singular supercuspidal representations \cite{2019arXiv191203274K}.
The construction of non-singular supercuspidal representations starts with a tame $k_{F}$-non-singular elliptic pair $(S, \theta)$, i.\,e., $S$ is a tame elliptic maximal torus of $G$, and $\theta$ is a character of $S(F)$ which satisfy the conditions in \cite[Definition~3.4.1]{2019arXiv191203274K}.
From the pair $(S, \theta)$, Kaletha constructed the tuple
\[\left((G^i)_{i=0}^{d}, y, (r_{i})_{i=0}^{d}, \kappa_{(S, \phi_{-1})}, (\phi_i)_{i=0}^{d}\right)\]
which satisfies the conditions {\bf{D1}}, {\bf{D2}}, {\bf{D3}}, and {\bf{D5}} in Section~\ref{secyu}, where $\kappa_{(S, \phi_{-1})}$ is a possibly reducible representation of $G^0(F)_{[y]}$ such that $\rho_{-1}\restriction_{G^0(F)_{y, 0}}$ is the inflation of a cuspidal representation of $G^0(F)_{y,0}/ G^0(F)_{y, 0+}$.
Hence, for each irreducible constituent $\rho_{-1}$ of $\kappa_{(S, \phi_{-1})}$, the tuple
\[\left((G^i)_{i=0}^{d}, y, (r_{i})_{i=0}^{d}, \rho_{-1}, (\phi_i)_{i=0}^{d}\right)\]
is a generic cuspidal $G$-datum.
Then, we obtain a supercuspidal representation of $G(F)$ from this datum by the construction of Yu.
An irreducible supercuspidal representation of $G(F)$ obtained in this way is called non-singular.

We explain the construction of the tuple
\[\left((G^i)_{i=0}^{d}, y, (r_{i})_{i=0}^{d}, \kappa_{(S, \phi_{-1})}, (\phi_i)_{i=0}^{d}\right)\]
from a tame $k_{F}$-non-singular elliptic pair $(S, \theta)$.

First, we consider a tame $k_{F}$-non-singular elliptic pair $(S, \theta)$ of depth-zero.
In this case, let $d=r_{0}=0$, $\phi_{0}$ be trivial, and $\phi_{-1}=\theta$. 

Let $[y]$ be the unique Frobenius-fixed point in $\mathcal{A}^{\text{red}}(S, F^u)$. According to \cite[Lemma~3.4.3]{MR4013740}, the point $[y]$ is a vertex of $\mathcal{B}^{\text{red}}(G, F)$.
We chose $y \in \mathcal{B}(G, F)$ whose projection on $\mathcal{B}^{\text{red}}(G, F)$ is $[y]$.

There exists a unique smooth integral model of $G$ whose group of $\mathcal{O}_{F^{u}}$- points equal to the stabilizer $G(F^{u})_{[y]}$ of $[y]$ in $G(F^{u})$.
Let $\mathsf{G}_{[y]}$ be the quotient of the special fiber of this model modulo its unipotent radical.
Then $\mathsf{G}_{[y]}$ is a smooth $k_F$-group scheme.
Let $\mathsf{G}_{[y]}^\circ$ be its neutral connected component, which is the reductive quotient of the special fiber of the parahoric group scheme of $G$ associated to the vertex $[y]$. We have $\mathsf{G}_{[y]}(k_F)= G(F)_{[y]}/G(F)_{y, 0+}$ and $\mathsf{G}_{[y]}^\circ(k_F)=G(F)_{y, 0}/G(F)_{y, 0+}$ \cite[Section~3]{2019arXiv191203274K}.
We also define the corresponding $k_F$-group schemes $\mathsf{S}$ and $\mathsf{S}^\circ$ satisfying $\mathsf{S}(k_F)=S(F)/S(F)_{0+}$ and $\mathsf{S}^\circ(k_F)=S(F)_{0}/S(F)_{0+}$.
Finally, we put $\mathsf{G}'_{y}=\mathsf{S} \cdot \mathsf{G}_{[y]}^\circ$. Then, Lang's theorem implies that $\mathsf{G}'_{[y]}(k_F)=S(F)G(F)_{y, 0}/G(F)_{y, 0+}$.
We regard $\theta$ as a character of $\mathsf{S}(k_F)$.

In \cite{2019arXiv191203274K}, Kaletha extended the theory of \cite{MR393266} for disconnected groups and defined the representation $\kappa_{(\mathsf{S}, \theta)}^{\mathsf{G}_{[y]}}$ of $\mathsf{G}_{[y]}(k_F)$ from the pair $(\mathsf{S}, \theta)$.
Then, we obtain the representation $\kappa_{(S, \theta)}$ of $G(F)_{[y]}$ by inflating $\kappa_{(\mathsf{S}, \theta)}^{\mathsf{G}_{[y]}}$ to $G(F)_{[y]}$.
We explain the construction of $\kappa_{(\mathsf{S}, \theta)}^{\mathsf{G}_{[y]}}$.

Let $\mathsf{U} \subset \mathsf{G}_{[y]}^\circ$ be the unipotent radical of a Borel subgroup of $\mathsf{G}_{[y]}^\circ$ containing $\mathsf{S}^\circ$, and let $\text{Fr}$ denote the Frobenius endomorphism of $\mathsf{G}_{[y]}$. We define the corresponding Deligne--Lusztig variety
\[Y_{\mathsf{U}}^{\mathsf{G}_{[y]}} = \{g\mathsf{U} \in \mathsf{G}_{[y]} / \mathsf{U} \mid g^{-1}\text{Fr}(g) \in \mathsf{U} \cdot \text{Fr}\left(\mathsf{U}\right)\}.\]
The group $\mathsf{G}_{[y]}(k_F)$ acts on $Y_{\mathsf{U}}^{\mathsf{G}_{[y]}}$ by the left multiplication and $\mathsf{S}(k_F)$ acts on $Y_{\mathsf{U}}^{\mathsf{G}_{[y]}}$ by the right multiplication, so these groups act on the $\ell$-adic cohomology group $H^{i}_{c} (Y_{\mathsf{U}}^{\mathsf{G}_{[y]}}, \overline{\mathbb{Q}_{\ell}})$.
We define
\[H^{i}_{c} (Y_{\mathsf{U}}^{\mathsf{G}_{[y]}}, \overline{\mathbb{Q}_{\ell}})_{\theta}= \{v\in H^{i}_{c} (Y_{\mathsf{U}}^{\mathsf{G}_{[y]}}, \overline{\mathbb{Q}_{\ell}}) \mid v\cdot s=\theta(s)v \ \forall s\in \mathsf{S}(k_F)\}.\]
Then, $\mathsf{G}_{[y]}(k_F)$ acts on $H^{i}_{c} (Y_{\mathsf{U}}^{\mathsf{G}_{[y]}}, \overline{\mathbb{Q}_{\ell}})_{\theta}$. 
According to \cite[Corollary~2.6.3]{2019arXiv191203274K}, this component vanishes for all but one $i$, namely $i= d(\mathsf{U})$, where $d(\mathsf{U})$ denotes the number of root hyperplanes separating the Weyl chambers of $\mathsf{U}$ and $\text{Fr}\left( \mathsf{U} \right)$ respectively.
We define the representation $\kappa_{(\mathsf{S}, \theta)}^{\mathsf{G}_{[y]}}$ of $\mathsf{G}_{[y]}(k_F)$ by the action of $\mathsf{G}_{[y]}(k_F)$ on $H^{d(\mathsf{U})}_{c} (Y_{\mathsf{U}}^{\mathsf{G}_{[y]}}, \overline{\mathbb{Q}_{\ell}})_{\theta}$.
\begin{remark}
\label{rmkg'}
Replacing $\mathsf{G}_{[y]}$ with $\mathsf{G}'_{[y]}$ in the construction above, we define the representation $\kappa_{(\mathsf{S}, \theta)}^{\mathsf{G}'_{[y]}}$ of $\mathsf{G}'_{[y]}(k_F)$.
Similarly, replacing $\mathsf{G}_{[y]}$ with $\mathsf{G}_{[y]}^\circ$, $\mathsf{S}$ with $\mathsf{S}^\circ$, and $\theta$ with $\theta^\circ := \theta\restriction_{\mathsf{S}^\circ}$, we define the representation $\kappa_{(\mathsf{S}^\circ, \theta^\circ)}^{\mathsf{G}_{[y]}^\circ}$ of $\mathsf{G}_{[y]}^\circ(k_F)$.
We note that $\kappa_{(\mathsf{S}^\circ, \theta^\circ)}^{\mathsf{G}_{[y]}^\circ}$ is the usual Deligne--Lusztig representation of $\mathsf{G}_{[y]}^\circ$ arising from the pair $(\mathsf{S}^\circ, \theta^\circ)$.
According to \cite[Corollary~2.6.4]{2019arXiv191203274K} and \cite[Remark~2.6.5]{2019arXiv191203274K}, $\kappa_{(\mathsf{S}, \theta)}^{\mathsf{G}_{[y]}}$ is isomorphic to the induced representation $\text{Ind}_{\mathsf{G}'_{[y]}(k_F)}^{\mathsf{G}_{[y]}(k_F)} \kappa_{(\mathsf{S}, \theta)}^{\mathsf{G}'_{[y]}}$ and the representation $\kappa_{(\mathsf{S}, \theta)}^{\mathsf{G}'_{[y]}}$ is obtained by endowing $\kappa_{(\mathsf{S}^\circ, \theta^\circ)}^{\mathsf{G}_{[y]}^\circ}$ with a structure of $\mathsf{G}'_{[y]}(k_F)$-representation.
These results are used later in Section~\ref{secfdns}.
\end{remark}
Next, we consider a tame $k_F$-non-singular elliptic pair $(S, \theta)$ of general depth.
From $\left(S, \theta \right)$, we obtain a sequence of twisted Levi subgroups
\[\overrightarrow{G}=\left(S=G^{-1}\subset G^0\subsetneq \ldots \subsetneq G^d=G\right)\]
in $G$ and a sequence of real numbers $\overrightarrow{r}=\left(0=r_{-1}, r_0, \ldots , r_d\right)$ as explained below (see \cite[3.6]{MR4013740}).
Let $E$ be the splitting field of $S$. For each positive real number $r$, we define
\[R_{r}=\{
\alpha\in R(G, S) \mid (\theta \circ N_{E/F} \circ \check{\alpha})(E_{r}^{\times})=1
\},\]
where $N_{E/F}$ denotes the norm map $S(E) \to S(F)$.
We also define $R_{r+}=\cap_{s>r} R_{s}$ for $r \ge 0$.
Let $r_{d-1} > r_{d-2} > \cdots > r_{0} >0$ be the breaks of this filtration i.\,e., the positive real numbers $r$ with $R_{r+} \neq R_{r}$.
We set in addition $r_{-1}=0$ and $r_{d}=\text{depth}(\theta)$.
For each $0\le i \le d$, we define $G^{i}$ be the connected reductive subgroup of $G$ with maximal torus $S$ and root system $R_{r_{i-1}+}$. We set $G^{-1}=S$. 
According to \cite[Lemma~3.6.1]{MR4013740}, the subgroups $G^i$ are twisted Levi subgroups of $G$.

We next recall the definition of a Howe factorization.
\begin{definition}[{\cite[Definition~3.6.2]{MR4013740}}]
A Howe factorization of $\left(S, \theta\right)$ is a sequences of characters $\overrightarrow{\phi}=\left(\phi_{-1}, \ldots, \phi_d\right)$, where $\phi_i$ is a character of $G^i(F)$ with the following properties.
\begin{enumerate}
\item \[\theta=\prod_{i=-1}^{d} \phi_i\restriction_{S(F)};\]
\item For all $0\le i \le d$, the character $\phi_{i}$ is trivial on the image of $G^{i}_{\text{sc}}(F)$, where $G^{i}_{\text{sc}}(F)$ denotes the simply connected cover of the derived subgroup of $G^{i}$;
\item For all $0 \le i <d$, $\phi_i$ is $G^{i+1}$-generic of depth $r_i$ relative to $y$ in the sense of \cite[Section.~9]{MR1824988}. For $i=d$, $\phi_{d}$ is trivial if $r_{d-1}=r_d$ and has depth $r_{d}$ otherwise. For $i=-1$, $\phi_{-1}$ is trivial if $G^0=S$ and otherwise satisfies $\phi_{-1}\restriction_{S(F)_{0+}} = 1$. 
\end{enumerate}
\end{definition}
According to \cite[Proposition~3.6.7]{MR4013740}, $\left(S, \theta\right)$ has a Howe factorization.
We take a Howe factorization $\overrightarrow{\phi}$.
Then, $(S, \phi_{-1})$ is a tame $k_F$-non-singular elliptic pair of $G^0$, and the character $\phi_{-1}$ is of depth-zero.
Applying the construction above for the pair $(S, \phi_{-1})$, we have the representation $\kappa_{(S, \phi_{-1})}$ of $G^0(F)_{[y]}$.
In this way, we obtain the tuple
\[\left((G^i)_{i=0}^{d}, y, (r_{i})_{i=0}^{d}, \kappa_{(S, \phi_{-1})}, (\phi_i)_{i=0}^{d}\right)\]
from a tame $k_F$-non-singular elliptic pair $(S, \theta)$ .

If the pair $(S, \theta)$ is regular in the sense of \cite[Definition~3.7.5]{MR4013740}, the representation $\kappa_{(S, \phi_{-1})}$ is irreducible, and the tuple
\[\left((G^i)_{i=0}^{d}, y, (r_{i})_{i=0}^{d}, \kappa_{(S, \phi_{-1})}, (\phi_i)_{i=0}^{d}\right)\]
is a generic cuspidal $G$-datum \cite[Proposition~3.7.8]{MR4013740}. Applying the construction of Yu to this datum, we obtain a regular supercuspidal representation $\pi_{(S, \theta)}$ and the formal degree of $\pi_{(S, \theta)}$ is computed by combining Proposition~\ref{thmA} and the dimension formula for the Deligne--Lusztig representations, which is done in \cite[Corollary~52]{2021arXiv210100658S}.

In our case, however, the representation $\kappa_{(S, \phi_{-1})}$ can be reducible. 
Hence, in order to obtain a non-singular supercuspidal representation, we have to take an irreducible constituent of $\kappa_{(S, \phi_{-1})}$.
We explain the description of the set $[\kappa_{(S, \phi_{-1})}]$ of irreducible constituents of $\kappa_{(S, \phi_{-1})}$ \cite[Section~3]{2019arXiv191203274K}.
We note that the set $[\kappa_{(S, \phi_{-1})}]$ dose not denote a multiset.

First, we assume that $\theta$ is of depth-zero. Let $N(S, G)(F)_{\theta}$ be the stabilizer of the pair $(S, \theta)$ in $G(F)$ by the conjugate action, and $N(\mathsf{S}, \mathsf{G}_{[y]})(k_F)_{\theta}$ be the stabilizer of the pair $(\mathsf{S}, \theta)$ in $\mathsf{G}_{[y]}(k_F)$ by the conjugate action. 
An element $G(F)$ that normalizes $S$ also normalizes $\mathcal{A}^{\text{red}}(S, F^{u})$ and acts on $\mathcal{A}^{\text{red}}(S, F^{u})$ Frobenius-equivalently, hence fixes the point $[y]$.
Therefore, we get the inclusion map $N(S, G)(F)_{\theta} \to G(F)_{[y]}$.
Hence, we obtain a natural map $p \colon N(S, G)(F)_{\theta} \to N(\mathsf{S}, \mathsf{G}_{[y]})(k_F)_{\theta}$.

Let $\text{Irr}(N(S, G)(F)_{\theta}, \theta)$ be the set of irreducible representations of $N(S, G)(F)_{\theta}$ whose restriction to $S(F)$ is $\theta$-isotypic, and let $\text{Irr}(N(\mathsf{S}, \mathsf{G}_{[y]})(k_F)_{\theta}, \theta)$ be the set of irreducible representations of $N(\mathsf{S}, \mathsf{G}_{[y]})(k_F)_{\theta}$ whose restriction to $\mathsf{S}(k_F)$ is $\theta$-isotypic.
\begin{lemma}
\label{irrbij}
The natural map $p \colon N(S, G)(F)_{\theta} \to N(\mathsf{S}, \mathsf{G}_{[y]})(k_F)_{\theta}$ induces a bijection
\[\Irr(N(S, G)(F)_{\theta}, \theta) \longleftrightarrow \Irr(N(\mathsf{S}, \mathsf{G}_{[y]})(k_F)_{\theta}, \theta).\] 
\end{lemma}
\begin{proof}
The image  of $S(F)$ via the natural map $p \colon N(S, G)(F)_{\theta} \to N(\mathsf{S}, \mathsf{G}_{[y]})(k_F)_{\theta}$ is equal to $\mathsf{S}(k_F)$, and \cite[Lemma~3.2.2]{2019arXiv191203274K} implies that the induced map
\[N(S, G)(F)_{\theta}/ S(F) \longrightarrow N(\mathsf{S}, \mathsf{G}_{[y]})(k_F)_{\theta}/\mathsf{S}(k_F)\]
is an isomorphism.
Since the image of $p$ contains $\mathsf{S}(k_F)$, this isomorphism implies that $p$ is a surjection.
Let $\ker{p}$ denote the kernel of $p$.
Then, $p$ indeces an isomorphism 
\[N(S, G)(F)_{\theta}/\ker{p} \simeq N(\mathsf{S}, \mathsf{G}_{[y]})(k_F)_{\theta},\]
and we can identify the set $\Irr(N(\mathsf{S}, \mathsf{G}_{[y]})(k_F)_{\theta}, \theta)$ with the subset of $\Irr(N(S, G)(F)_{\theta}, \theta)$ consisting of the the representations whose restriction to $\ker{p}$ is trivial.

On the other hand, since $p$ induces an isomorphism 
\[N(S, G)(F)_{\theta}/ S(F) \longrightarrow N(\mathsf{S}, \mathsf{G}_{[y]})(k_F)_{\theta}/\mathsf{S}(k_F),\]
$\ker{p}$ is contained in $S(F)$.
Hence, $\ker{p}$ is equal to the kernel of the restriction $p\restriction_{S(F)} \colon S(F) \to \mathsf{S}(k_F) = S(F)/S(F)_{0+}$, which is equal to $S(F)_{0+}$.
Since $\theta$ is trivial on $S(F)_{0+}$, we conclude that every element in $\Irr(N(S, G)(F)_{\theta}, \theta)$ is trivial on $\ker{p}$. Therefore, we obtain a bijection
\[\Irr(N(S, G)(F)_{\theta}, \theta) \longleftrightarrow \Irr(N(\mathsf{S}, \mathsf{G}_{[y]})(k_F)_{\theta}, \theta).\] 
\end{proof}
In \cite[2.7]{2019arXiv191203274K}, Kaletha defined an action $R_{\mathsf{U}}^{\mathsf{G}_{[y]}, \epsilon}$ of $N(\mathsf{S}, \mathsf{G}_{[y]})(k_F)_{\theta}$ on $H^{d(\mathsf{U})}_{c} (Y_{\mathsf{U}}^{\mathsf{G}_{[y]}}, \overline{\mathbb{Q}_{\ell}})_{\theta}$ whose restriction to $\mathsf{S}(k_F)$ is $\theta^{-1}$-isotypic.
%The action $R_{\mathsf{U}}^{\mathsf{G}_{[y]}, \epsilon}$ depends on the choice of coherent splittings $\epsilon$ explained in \cite[2.4]{2019arXiv191203274K}.
The action $R_{\mathsf{U}}^{\mathsf{G}_{[y]}, \epsilon}$ commutes with the action of $\mathsf{G}_{[y]}(k_F)$ by $\kappa_{(\mathsf{S}, \theta)}^{\mathsf{G}_{[y]}}$. Therefore, we obtain a representation of $\mathsf{G}_{[y]}(k_F) \times N(\mathsf{S}, \mathsf{G}_{[y]})(k_F)_{\theta}$ on $H^{d(\mathsf{U})}_{c} (Y_{\mathsf{U}}^{\mathsf{G}_{[y]}}, \overline{\mathbb{Q}_{\ell}})_{\theta}$. Let $\kappa_{(\mathsf{S}, \theta)}^{\mathsf{G}_{[y]}, \epsilon}$ denote this $\mathsf{G}_{[y]}(k_F) \times N(\mathsf{S}, \mathsf{G}_{[y]})(k_F)_{\theta}$-representation.
For $\rho \in \text{Irr}(N(\mathsf{S}, \mathsf{G}_{[y]})(k_F)_{\theta}, \theta)$, we define 
\[\kappa_{(\mathsf{S}, \theta, \rho)}^{\mathsf{G}_{[y]}, \epsilon}= \text{Hom}_{N(\mathsf{S}, \mathsf{G}_{[y]})(k_F)_{\theta}} (\check{\rho}, \kappa_{(\mathsf{S}, \theta)}^{\mathsf{G}_{[y]}, \epsilon}),\]
where $\check{\rho}$ denotes the contragradient representation of $\rho$.
It is a representation of $\mathsf{G}_{[y]}(k_F)$.
\begin{proposition}[{\cite[Theorem~2.7.7]{2019arXiv191203274K}}]
The map $\rho \mapsto \kappa_{(\mathsf{S}, \theta, \rho)}^{\mathsf{G}_{[y]}, \epsilon}$ is a bijection
\[\Irr(N(\mathsf{S}, \mathsf{G}_{[y]})(k_F)_{\theta}, \theta) \longrightarrow \left[\kappa_{(\mathsf{S}, \theta)}^{\mathsf{G}_{[y]}}\right], \]
where $\left[\kappa_{(\mathsf{S}, \theta)}^{\mathsf{G}_{[y]}}\right]$ denotes the set of irreducible constituents in $\kappa_{(\mathsf{S}, \theta)}^{\mathsf{G}_{[y]}}$. 
Moreover, the multiplicity of $\kappa_{(\mathsf{S}, \theta, \rho)}^{\mathsf{G}_{[y]}, \epsilon}$ in $\kappa_{(\mathsf{S}, \theta)}^{\mathsf{G}_{[y]}}$ is equal to the dimension of $\rho$.
\end{proposition}
For $\rho \in \text{Irr}(S, G)(F)_{\theta}, $ let $\kappa_{(S, \theta, \rho)}^{\epsilon}$ be the representation of $G(F)_{[y]}$ obtained by the inflation of $\kappa_{(\mathsf{S}, \theta, \rho)}^{\mathsf{G}_{[y]}, \epsilon}$. Here, we regard $\rho$ as an element of $\text{Irr}(N(\mathsf{S}, \mathsf{G}_{[y]})(k_F)_{\theta}, \theta)$ by the bijection
\[\text{Irr}(N(S, G)(F)_{\theta}, \theta) \longleftrightarrow \text{Irr}(N(\mathsf{S}, \mathsf{G}_{[y]})(k_F)_{\theta}, \theta).\] 
of Lemma~\ref{irrbij}.
\begin{corollary}
\label{depthzero}
The map $\rho \mapsto \kappa_{(S, \theta, \rho)}^{\epsilon}$ is a bijection
\[\Irr(N(S, G)(F)_{\theta}, \theta) \longrightarrow \left[\kappa_{(S, \theta)}\right].\]
Moreover, the multiplicity of $\kappa_{(S, \theta, \rho)}^{\epsilon}$ in $\kappa_{(S, \theta)}$ is equal to the dimension of $\rho$.
\end{corollary}
We now consider the general case.
Let $(S, \theta)$ be a tame $k_F$-non-singular elliptic pair of general depth.
From the pair $(S, \theta)$, we obtain a sequences of twisted Levi subgroups $\overrightarrow{G}$ and a sequence of real numbers $\overrightarrow{r}$, and we take a Howe decomposition $\overrightarrow{\phi}$ of $(S, \theta)$ as above.
As in the case of depth-zero, we define $N(S, G)(F)_{\theta}$ be the stabilizer of the pair $(S, \theta)$ in $G(F)$ by the conjugate action.
We also define the groups $N(S, G^0)(F)_{\phi_{-1}}$ and $N(S, G^0)(F)_{\theta}$ similarly.
According to \cite[Lemma~3.4.5]{2019arXiv191203274K}, the natural inclusion $G^0(F) \to G(F)$ gives the identification
\[N(S, G^0)(F)_{\phi_{-1}} = N(S, G^0)(F)_{\theta} = N(S, G)(F)_{\theta}.\]
Let $\Irr(N(S, G)(F)_{\theta}, \theta)$ be the set of irreducible representations of $N(S, G)(F)_{\theta}$ whose restriction to $S(F)$ is $\theta$-isotypic, and let $\Irr(N(S, G^0)(F)_{\phi_{-1}}, \phi_{-1})$ be the set of irreducible representations of $N(S, G^0)(F)_{\phi_{-1}} = N(S, G)(F)_{\theta}$ whose restriction to $S(F)$ is $\phi_{-1}$-isotypic.
Put $\delta_{0}= \prod_{i=0}^{d} \phi_{i}^{-1}\restriction_{G^0(F)}$. Then, $\rho \mapsto \rho_{-1}:=\delta_{0} \otimes \rho$  is a bijection
\[\Irr(N(S, G)(F)_{\theta}, \theta) \longleftrightarrow \Irr(N(S, G^0)(F)_{\phi_{-1}}, \phi_{-1}).\]
Therefore, Corollary~\ref{depthzero} implies the following result.
\begin{corollary}
The map $\rho \mapsto \kappa_{(S, \phi_{-1}, \rho_{-1})}^{\epsilon}$ is a bijection
\[\Irr(N(S, G)(F)_{\theta}, \theta) \longrightarrow \left[\kappa_{(S, \phi_{-1})}\right].\]
Moreover, the multiplicity of $\kappa_{(S, \phi_{-1}, \rho_{-1})}^{\epsilon}$ in $\kappa_{(S, \phi_{-1})}$ is equal to the dimension of $\rho$.
\end{corollary}

Let $\rho\in \Irr(N(S, G)(F)_{\theta}, \theta)$. 
We define the non-singular representation $\pi_{(S, \theta, \rho)}^{\epsilon}$ of $G(F)$ be the representation obtained from the generic cuspidal $G$-datum
\[\left((G^i)_{i=0}^{d}, y, (r_{i})_{i=0}^{d}, \kappa_{(S, \phi_{-1}, \rho_{-1})}^{\epsilon}, (\phi_i)_{i=0}^{d}\right)\]
by the twisted Yu's construction defined in \cite{FKS2021}.
\begin{remark}
\label{twisted}
We need not to concern ourselves with the precise definition of the twisted Yu's construction, but only need to know that the supercuspidal representation obtained by twisted Yu's construction from a generic cuspidal $G$-datum
\[\Psi=(\overrightarrow{G}, y, \overrightarrow{r}, \rho_{-1}, \overrightarrow{\phi})\]
is the representation $\ind_{K^d}^{G(F)} (\rho_{d} \otimes e)$, where $e$ is a sign character of $K^d$ and $\ind_{K^d}^{G(F)} (\rho_{d} \otimes e)$ denotes the compactly induced representation.
According to \cite[Lemma~18]{2021arXiv210100658S}, the formal degree of the representation $\pi_{\Psi}$ is equal to
\[\frac{\dim({\rho_{d}})}{\vol(K^{d}/(K^{d}\cap A))}\]
and the formal degree of the representation $\ind_{K^d}^{G(F)} (\rho_{d} \otimes e)$ is equal to
\[\frac{\dim({\rho_{d} \otimes e})}{\vol(K^{d}/(K^{d}\cap A))},\]
where $\vol(K^d/(K^d \cap A))$ denotes the volume of $K^d/(K^d \cap A)$ with respect to the Haar measure $\mu$ on $G(F)/A(F)$ defined in Section~\ref{secfdc}.
Since $\dim(\rho_{d})=\dim(\rho_{d}\otimes e)$, the formal degree of two representations are equal.
\end{remark}
We also define the supercuspidal representation $\pi_{(S, \theta)}$ of $G(F)$, which is possibly reducible, as the representation obtained form the datum
\[\left((G^i)_{i=0}^{d}, y, (r_{i})_{i=0}^{d}, \kappa_{(S, \phi_{-1})}, (\phi_i)_{i=0}^{d}\right)\]
by the twisted Yu's construction. Then the set $[\pi_{(S, \theta)}]$ of irreducible constituents of $\pi_{(S, \theta)}$ is described as follows.
\begin{proposition}[{\cite[Corollary~3.4.7]{2019arXiv191203274K}}]
\label{bij1}
The map $\rho \mapsto \pi_{(S, \theta, \rho)}^{\epsilon}$ is a bijection
\[\Irr(N(S, G)(F)_{\theta}, \theta) \longrightarrow [\pi_{(S, \theta)}].\]
Moreover, the multiplicity of $\pi_{(S, \theta, \rho)}^{\epsilon}$ in $\pi_{(S, \theta)}$ is equal to the dimension of $\rho$.
\end{proposition}
\section{Torally wild $L$-parameters and torally wild $L$-packet data}
\label{seclpara}
In \cite{2019arXiv191203274K}, Kaletha defined the notions of torally wild $L$-parameters and torally wild $L$-packet data.
First, we recall the definitions of these notions.
\begin{definition}[{\cite[Definition~4.1.2]{2019arXiv191203274K}}]
A discrete $L$-parameter
\[\varphi \colon W_F \longrightarrow {^{L}G}\]
is called torally wild if the projection of $\varphi(P_F)$ on $\widehat{G}$ is contained in a maximal torus of $\widehat{G}$. 
\end{definition}
If $p$ dose not divide the order of the Weyl group of $G$, any discrete $L$-parameter $\varphi: W_F \to {^{L}G}$ is torally wild \cite[Lemma~4.1.3]{2019arXiv191203274K}.
\begin{definition}[{\cite[Definition~4.1.4]{2019arXiv191203274K}}]
A torally wild $L$-packet datum is a tuple $(S, \widehat{j}, \chi_0, \theta)$, where
\begin{enumerate}
\item $S$ is a torus of dimension equal to the absolute rank of $G$, defined over $F$ and splits over a tamely ramified extension of $F$;
\item $\widehat{j}\colon \widehat{S}\to \widehat{G}$ is an embedding of complex reductive groups whose $\widehat{G}$-conjugacy class is $\Gamma_F$-stable, i.\,e., for any $\gamma \in \Gamma_F$, the embedding $\gamma \circ \widehat{j} \circ \gamma^{-1} \colon \widehat{S} \to \widehat{G}$ is $\widehat{G}$-conjugate to $\widehat{j}$;
\item$\chi_0=(\chi_{\alpha_0})_{\alpha_0\in R(G, S^0)}$ is tamely ramified $\chi$-data for $R(G,S^0)$, as explained below;
\item $\theta\colon S(F) \to \mathbb{C}^{\times}$ is a character,
\end{enumerate}
subject to the condition that $(S, \theta)$ is a tame $F$-non-singular elliptic pair in the sense of \cite[Definition~3.4.1]{2019arXiv191203274K}.
\end{definition}
We explain the third point. 
Let $G'$ be the quasi-split inner form of $G$.
As explained in \cite[5.1]{MR4013740}, the embedding $\widehat{j}$ determines a $\Gamma_{F}$-stable conjugacy class of embeddings $j\colon S \to G'$. By choosing an embedding $j\colon S \to G'$, and pulling back the root system $R(G', jS)$, 
we obtain a $\Gamma_F$-invariant root system $R(G, S)\subset X^{*}(S)$, which dose not depend on the choice of $j$.
Let $E$ be the splitting field of $S$. For each positive real number $r$, we define
\[R_{r}=\{
\alpha\in R(G, S) \mid (\theta \circ N_{E/F} \circ \check{\alpha})(E_{r}^{\times})=1
\},\]
where $N_{E/F}$ denotes the norm map $S(E) \to S(F)$.
We also define $R_{r+}=\cap_{s>r} R_{s}$ for $r \ge 0$.
Let $S^0$ be the connected component of the intersection of the kernels of all elements of $R_{0+}$ and $R(G,S^0)$ be the image of $R(G, S)\backslash R_{0+}$ under the restriction map $X^{*}(S) \to X^{*}(S^0)$. 

We now recall the definition of $\chi$-data (see \cite[2.5]{MR909227}).
Let $R=R(G, S)$ or $R(G, S^0)$.
For $\alpha \in R$, let $\Gamma_{\alpha}$ be the stabilizer of $\alpha$ in $\Gamma_{F}$ and $F_{\alpha}$ be the corresponding fixed subfield of $F^{\text{sep}}$.
We also define $\Gamma_{\pm \alpha}$ be the stabilizer of the set $\{\alpha, -\alpha\}$ and $F_{\pm \alpha}$ be the corresponding fixed subfield of $F^{\text{sep}}$.
We say that $\alpha$ is symmetric if $F_{\alpha}/F_{\pm \alpha}$ is a quadratic extension, and asymmetric if $F_{\alpha}=F_{\pm \alpha}$.
Let $R_{\text{sym}}$ denote the set of symmetric roots $\alpha$ in $R$.
For $\alpha \in R_{\text{sym}}$, we say that $\alpha$ is unramified (resp.\,ramified) if the extension $F_{\alpha}/F_{\pm \alpha}$ is unramified (resp.\,ramified).

A set of $\chi$-data for $R$ consists of characters $\chi_{\alpha}\colon F_{\alpha}^{\times} \to \mathbb{C}^{\times}$, one for each $\alpha \in R$, having the properties $\chi_{-\alpha}=\chi_{\alpha}^{-1}$, $\chi_{\gamma(\alpha)}=\chi_{\alpha} \circ \gamma^{-1}$ for each $\gamma \in \Gamma_{F}$, and $\chi_{\alpha}\restriction_{F_{\pm \alpha}^{\times}}$ is the non-trivial quadratic character of $F_{\pm \alpha}^{\times}$ which is trivial on the image of the norm map $F_{\alpha}^{\times} \to F_{\pm\alpha}^{\times}$ for $\alpha \in R_{\text{sym}}$.
A set of $\chi$-data $\chi=(\chi_{\alpha})_{\alpha\in R}$ for $R$ is called unramified (resp.\,tamely ramified) if each character $\chi_{\alpha}$ is unramified (resp.\,tamely ramified), i.\,e., trivial on $\mathcal{O}_{F_{\alpha}}^{\times}$. (resp.\,trivial on $(F_{\alpha})_{0+}^{\times}$). 
We also say that a set of $\chi$-data $\chi=(\chi_{\alpha})_{\alpha\in R}$ for $R$ is minimally ramified, if $\chi_{\alpha}=1$ for asymmetric $\alpha$, $\chi_{\alpha}$ is unramified for unramified symmetric $\alpha$, and $\chi_{\alpha}$ is tamely ramified for ramified symmetric $\alpha$. 

%\begin{definition}[{\cite[Definition~4.1.6]{2019arXiv191203274K}}]
%A morphism $(S, \widehat{j}, \chi_0, \theta) \to (S', \widehat{j}', \chi'_0, \theta')$ of totally wild supercuspidal $L$-packet data is a tuple $(\iota, g, \zeta_0)$, where
%\begin{enumerate}
%\item $\iota\colon S\to S'$ is an isomorphism of $F$-tori;
%\item $g\in \hat{G}$;
%\item $\zeta_0=(\zeta_{\alpha'_{0}})_{\alpha'_{0}}$ is a set of $\zeta$-data for $R(S'_{0}, F)$ in the sense of \cite[Definition~4.6.4]{MR4013740}
%\end{enumerate}
%such that $\widehat{j}\circ \hat{\iota}=\text{Ad}(g) \circ \widehat{j}'$,$\chi_{\alpha'_{0}\circ \iota}=\chi'_{\alpha'_{0}}$, and $(\zeta_{S'}^{-1}\cdot \theta')\circ \iota=\theta$, where $\zeta=\text{inf}\zeta_{0}$ in the sense of and $\zeta_{S'}$ is the character of $S'(F)$ corresponding to $\zeta$ as in \cite[Definition~4.6.5]{MR4013740}.
%\end{definition}
Kaletha defined the notion of morphisms between torally wild $L$-packet data \cite[Definition~4.1.6]{2019arXiv191203274K} and gave a bijection between the set of $\widehat{G}$-conjugacy classes of torally wild $L$-parameters and the set of isomorphism classes of torally wild $L$-packet data \cite[Proposition~4.1.8]{2019arXiv191203274K}. 
We recall the way to construct torally wild $L$-parameters from torally wild $L$-packet data briefly.

Let $(S, \widehat{j}, \chi_0, \theta)$ be a torally wild $L$-packet datum.
From the $\chi$-data $\chi_0$ for $R(G, S^0)$, we obtain $\chi$-data $\chi=(\chi_{\alpha})_{\alpha}$ for $R(G, S)$ as follows.
For $\alpha \in R(G, S) \backslash R_{0+}$, we let $\chi_{\alpha}=\chi_{\alpha_0} \circ N_{F_{\alpha}/F_{\alpha_{0}}}$, where $\alpha_0$ denotes the image of $\alpha$ via the restriction map $X^{*}(S) \to X^{*}(S^0)$, and $N_{F_{\alpha}/F_{\alpha_0}}$ denotes the norm map $F_{\alpha}^{\times} \to F_{\alpha_0}^{\times}$.
For $\alpha \in R_{0+}$, we let $\chi_{\alpha}$ be trivial if $\alpha$ is asymmetric and be the non-trivial unramified quadratic character if $\alpha$ is unramified symmetric.
We note that since $(S, \theta)$ is a tame $F$-non-singular elliptic pair, the action of $I_F$ on $R_{0+}$ preserves a set of positive roots. Hence, every symmetric root $\alpha \in R_{0+}$ is unramified.

We extend $\widehat{j}$ to an $L$-embedding $^{L}\!j\colon {^{L}S} \to {^{L}G}$ by using this $\chi$-data as in \cite[6.1]{2019arXiv190705173K} (see also \cite[2.6]{MR909227}).
Then, let $\varphi={^{L}\!j}\circ \varphi_{S}$, where $\varphi_{S}\colon W_F \to {^{L}S}$ is the $L$-parameter attached to the character $\theta$ via the local Langlands correspondence for tori \cite[Theorem~7.5]{MR2508725}.
In this way, we obtain from the tuple $(S, \widehat{j}, \chi_0, \theta)$ a torally wild $L$-parameter $\varphi$.

Finally, we explain the construction of the $L$-packet associated with a torally wild $L$-parameter $\varphi$ \cite[4.2]{2019arXiv191203274K}. Let $(S, \widehat{j}, \chi_0, \theta)$ be the torally wild $L$-packet datum whose isomorphism class corresponds to the $\widehat{G}$-conjugacy class of $\varphi$. 
As explained in \cite[4.2]{2019arXiv191203274K}, we may assume that the $\chi$-data for $R(G, S)$ obtained from $\theta$ as in \cite[3.5]{2019arXiv191203274K} is equal to the $\chi$-data for $R(G, S)$ obtained from $\chi_0$ as above. 
As explained in \cite[5.1]{MR4013740}, the embedding $\widehat{j}$ determines the class of admissible embeddings $j \colon S \to G$.
Then, we define the $L$-packet $\prod_{\varphi}(G)$ as the union of $[\pi_{(jS, \theta_j)}]$, where $j\colon S \to G$ ranges over the $G(F)$-conjugacy classes of admissible embeddings defined over $F$, and $\theta_j$ is a character of $jS(F)$ defined by $\theta_j(j(s))=\theta(s)$ for $s\in S(F)$. The resulting $L$-packet $\prod_{\varphi}(G)$ dose not depend on the choice of torally wild $L$-packet datum $(S, \widehat{j}, \chi_0, \theta)$ which satisfies the above condition. 
\section{Endoscopy}
\label{secend}
Let $\varphi$ be a torally wild $L$-parameter which corresponds to the torally wild $L$-packet datum $(S, \widehat{j}, \chi_0, \theta)$.
We regard $Z$ as a subgroup of $S$ by using an admissible embedding $j \colon S \to G$.
We note that this structure dose not depend on the choice of an embedding $j$ (see \cite[5.1]{MR4013740}).
Let $[\widehat{\overline{S}}]^{+}$ be the preimage of $\widehat{S}^{\Gamma_F}$ via the map $\widehat{S/Z} \to \widehat{S}$.
We also define the group $S_{\varphi}^{+}$ be the preimage of $S_{\varphi}$ via the map $\widehat{G/Z} \to \widehat{G}$.
For these groups, we denote by $\pi_0(*)$ the groups of connected components.

In the previous section, we define the $L$-packet $\prod_{\varphi}(G)$ associated with a torally wild $L$-parameter $\varphi$. 
In \cite[4.4, 4.5]{2019arXiv191203274K}, Kaletha constructed a bijection between $\prod_{\varphi}(G)$ and a set of irreducible representations of $\pi_0(S_{\varphi}^{+})$ satisfying some conditions. 
\begin{remark}
The formulation of the local Langlands correspondence in \cite{2019arXiv191203274K} uses the group $\pi_0(S_{\varphi}^{+})$ to parametrize the elements of $L$-packets.
On the other hand, the formulation of the local Langlands correspondence of \cite{MR2217572}, on which Conjecture~\ref{fdc} depends, uses the group $\pi_0(S_{\varphi}^{\text{sc}})$ as explained in Section~\ref{secfdc}.
However, in \cite[4.6]{MR3815567}, Kaletha gave a dimension-preserving bijection between sets of irreducible representations of two different $S$-groups which is compatible with both formulations of the local Langlands correspondence, and proved that the local Langlands correspondence in \cite{2019arXiv191203274K} implies the one in \cite{MR2217572}. 
Therefore, it is enough to prove Conjecture~\ref{fdc} for the formulation in \cite{2019arXiv191203274K}.
\end{remark}
The bijection above is given by combining the bijections
\[[\pi_{(jS, \theta_j)}] \longleftrightarrow \text{Irr}(\pi_0(S_{\varphi}^{+}), \eta)\]
for all admissible embeddings $j\colon S \to G$,
where $\eta$ is the character of $\pi_0([\widehat{\overline{S}}]^{+})$ which is determined by $j$ as in \cite[4.4]{2019arXiv191203274K}, and $\text{Irr}(\pi_0(S_{\varphi}^{+}), \eta)$ denotes the set of irreducible representations of $\pi_0(S_{\varphi}^{+})$ whose restriction to $\pi_0([\widehat{\overline{S}}]^{+})$ contains $\eta$.
Here, we regard $\pi_0([\widehat{\overline{S}}]^{+})$ as a subgroup of $\pi_0(S_{\varphi}^{+})$ by using $\widehat{j}$ (see \cite[Corollary~4.3.4]{2019arXiv191203274K}).

We now explain the bijection
\[[\pi_{(jS, \theta_j)}] \longleftrightarrow \text{Irr}(\pi_0(S_{\varphi}^{+}), \eta)\]
in \cite{2019arXiv191203274K}.
First, we assume that $\varphi$ is essentially of depth-zero in the sense of \cite[4.5]{2019arXiv191203274K}.

Proposition~\ref{bij1} implies that the map $\rho \mapsto \pi_{(jS, \theta_j, \rho)}^{\epsilon}$ is a bijection
\begin{align}
\label{bijc1}
\text{Irr}(N(jS, G)(F)_{\theta_j}, \theta_j) \longleftrightarrow [\pi_{(jS, \theta_j)}].
\end{align}

Let $\Box_2$ be the pushout of $\theta_{j} \colon jS(F) \to \mathbb{C}^{\times}$ and the inclusion $jS(F) \to N(jS, G)(F)_{\theta_j}$.
Then, we obtain the extension
\[1 \to \mathbb{C}^{\times} \to \Box_2 \to \frac{N(jS, G)(F)_{\theta_j}}{jS(F)} \to 1,\]
and \cite[Lemma~C.5]{2019arXiv191203274K} implies that there exists a natural dimension-preserving bijection
\begin{align}
\label{bijc2}
\text{Irr}(N(jS, G)(F)_{\theta_j}, \theta_j) \longleftrightarrow \text{Irr}(\Box_2, \text{id}),
\end{align}
where $\text{Irr}(\Box_2, \text{id})$ denotes the set of irreducible representations of $\Box_2$ whose restriction to $\mathbb{C}^{\times}$ is $\text{id}$-isotypic.

On the other hand, according to \cite[Corollary~4.3.4]{2019arXiv191203274K}, there exists an exact sequence
\[1 \to \pi_0([\widehat{\overline{S}}]^{+}) \to \pi_0(S_{\varphi}^{+}) \to \Omega(S, G)(F)_{\theta} \to 1,\]
where $\Omega(S, G)$ denotes the $\Gamma_F$-invariant subgroup of the automorphism group of $S$ defined in \cite[5.1]{MR4013740}, and $\Omega(S, G)(F)_{\theta}$ denotes the stabilizer of $\theta$ in $\Omega(S, G)(F)$.
Let $\pi_0(S_{\varphi}^{+})_{\eta}$ be the stabilizer of $\eta$ in $\pi_0(S_{\varphi}^{+})$ by the conjugate action, and $\Omega(S, G)(F)_{\theta, \eta}$ be the stabilizer of $\eta$ in $\Omega(S, G)(F)_{\theta}$ by the conjugate action. Then, we obtain an exact sequence
\[1 \to \pi_0([\widehat{\overline{S}}]^{+}) \to \pi_0(S_{\varphi}^{+})_{\eta} \to \Omega(S, G)(F)_{\theta, \eta} \to 1\]
from the exact sequence above.
We define $\Box_1$ be the pushout of $\eta\colon \pi_0([\widehat{\overline{S}}]^{+}) \to \mathbb{C}^{\times}$ and the map $\pi_0([\widehat{\overline{S}}]^{+}) \to \pi_0(S_{\varphi}^{+})_{\eta}$ above.
Then, we obtain the extension
\[1\to \mathbb{C}^{\times} \to \Box_1 \to \Omega(S, G)(F)_{\theta, \eta} \to 1,\]
and \cite[Lemma~C.5]{2019arXiv191203274K} implies that there exists a natural bijection
\begin{align}
\label{bijc3}
\text{Irr}(\pi_0(S_{\varphi}^{+}), \eta) \longleftrightarrow \text{Irr}(\Box_1, \text{id}),
\end{align}
where $\text{Irr}(\Box_1, \text{id})$ denotes the set of irreducible representations of $\Box_1$ whose restriction to $\mathbb{C}^{\times}$ is $\text{id}$-isotypic.
The bijection \eqref{bijc3} is a composition of the natural dimension-preserving bijection
\[\text{Irr}(\pi_0(S_{\varphi}^{+})_{\eta}, \eta) \longleftrightarrow \text{Irr}(\Box_1, \text{id})\]
and the bijection 
\[\text{Irr}(\pi_0(S_{\varphi}^{+})_{\eta}, \eta) \longleftrightarrow \text{Irr}(\pi_0(S_{\varphi}^{+}), \eta)\]
given by the induced representation (see the proof of \cite[Lemma~C.5]{2019arXiv191203274K}).
Here, $\text{Irr}(\pi_0(S_{\varphi}^{+})_{\eta}, \eta)$ denotes the set of irreducible representations of $\pi_0(S_{\varphi}^{+})_{\eta}$ whose restriction to $\pi_0([\widehat{\overline{S}}]^{+})$ is $\eta$-isotypic.
Hence, for $\rho \in \text{Irr}(\pi_0(S_{\varphi}^{+}), \eta)$, the dimension of the representation of $\Box_1$ which corresponds to $\rho$ via the bijection \eqref{bijc3} is
\[\abs{\pi_0(S_{\varphi}^{+})/\pi_0(S_{\varphi}^{+})_{\eta}}^{-1} \cdot \dim(\rho).\]

Moreover, \cite[Lemma~3.4.10]{MR4013740} and \cite[Lemma~E.1]{2019arXiv191203274K} imply that $j$ gives an isomorphism
\[\Omega(S, G)(F)_{\theta, \eta} \longrightarrow N(jS, G)(F)_{\theta_j}/jS(F),\]
and \cite[Proposition~4.5.1]{2019arXiv191203274K} implies that 
the extensions
\begin{align*}
&1 \to \mathbb{C}^{\times} \to \Box_2 \to \frac{N(jS, G)(F)_{\theta_j}}{jS(F)}\simeq \Omega(S, G)(F)_{\theta, \eta} \to 1\\
&1\to \mathbb{C}^{\times} \to \Box_1 \to \Omega(S, G)(F)_{\theta, \eta} \to 1
\end{align*}
above are isomorphic.
Therefore, there exists a dimension-preserving bijection 
\begin{align}
\label{bijc4}
\text{Irr}(\Box_1, \text{id}) \longleftrightarrow \text{Irr}(\Box_2, \text{id}).
\end{align}
By combining the bijections~\eqref{bijc1}, \eqref{bijc2}, \eqref{bijc3}, \eqref{bijc4}, we obtain a bijection
\[[\pi_{(jS, \theta_j)}] \longleftrightarrow \text{Irr}(\pi_0(S_{\varphi}^{+}), \eta).\]

We now consider the general case.
Recall that we obtain a sequence of twisted Levi subgroups
\[\overrightarrow{G}=\left(jS=G^{-1}\subset G^0\subsetneq \ldots \subsetneq G^d=G\right)\]
from the pair $(jS, \theta_j)$.
As explained in \cite[4.4]{2019arXiv191203274K}, $\varphi$ is decomposed as $\varphi= ^{L}\!j_{G^0, G} \circ \varphi_{G^0}$, where $\varphi_{G^0} \colon W_F \to {^{L}G^0}$ is the torally wild $L$-parameter of $G^0$ corresponding to the torally wild $L$-packet datum $(S, \widehat{j}, \emptyset, \theta)$, which is essentially of depth-zero, and $^{L}\!j_{G^0, G}\colon {^{L}G^0} \to {^{L}G}$ is the extension of $\widehat{G^0} \to \widehat{G}$ obtained from the $\chi$-data $\chi_0$ for $R(G, S^0)$ as in \cite[6.1]{2019arXiv190705173K}. 
Moreover, the $L$-embedding $^{L}\!j_{G^0, G}$ induces the identification $S_{\varphi}^{+}= S_{\varphi_{G^0}}^{+}$. In particular, there exists a canonical dimension-preserving bijection
\begin{align}
\label{bijc5}
\text{Irr}(\pi_0(S_{\varphi}^{+}), \eta) \longleftrightarrow \text{Irr}(\pi_0(S_{\varphi_{G^0}}^{+}), \eta).
\end{align}
On the other hand, regarding $(jS, \theta_j)$ as a tame $F$-non-singular elliptic pair of $G^0$, we obtain the supercuspidal representation $\pi_{(jS, \theta_j)}^{G^0}$ of $G^0(F)$. 
According to Proposition~\ref{bij1}, we obtain a bijection
\[\text{Irr}(N(jS, G^0)(F)_{\theta_j}, \theta_j) \longleftrightarrow [\pi_{(jS, \theta_j)}^{G^0}].\]
Since $N(jS, G)(F)_{\theta_j}=N(jS, G^0)(F)_{\theta_j}$ \cite[Lemma~3.6.5]{MR4013740}, we obtain a bijection
\begin{align}
\label{bijc6}
[\pi_{(jS, \theta_j)}^{G^0}] \longleftrightarrow \text{Irr}(N(jS, G^0)(F)_{\theta_j}, \theta_j) \longleftrightarrow \text{Irr}(N(jS, G)(F)_{\theta_j}, \theta_j)\longleftrightarrow [\pi_{(jS, \theta_j)}].
\end{align}
Combining the bijections~\eqref{bijc5}, \eqref{bijc6} and the bijection
\[[\pi_{(jS, \theta_j)}^{G^0}] \longleftrightarrow \text{Irr}(\pi_0(S_{\varphi_{G^0}}^{+}), \eta)\]
obtained from depth-zero case, we obtain a bijection
\begin{align}
\label{bijc7}
[\pi_{(jS, \theta_j)}] \longleftrightarrow \text{Irr}(\pi_0(S_{\varphi}^{+}), \eta).
\end{align}
for general case.

For $\rho \in \text{Irr}(\pi_0(S_{\varphi}^{+}), \eta)$, we write $\pi_{\rho}$ for the element in $[\pi_{(jS, \theta_j)}]$ which corresponds to $\rho$ via the bijection \eqref{bijc7}. 
\section{Key lemma}
\label{seclem}
In this section, we prove a variant of \cite[Proposition~B.3]{2019arXiv191203274K} to compare the dimension of $\rho$ and the formal degree of $\pi_{\rho}$.

Contrary to the conventions of the rest of the paper, in this section only, we use the notations below.
Let $G$ be a locally profinite group, $H \subset G$ be an open normal subgroup of finite index. We assume that $G/H$ is abelian. Let $N \subset G$ be a closed subgroup, write $N_H=N \cap H$, and let $S \subset N_H$ be an abelian open normal subgroup of  $N$ of finite index.

The group $N$ acts on $G$ by the conjugate action, and we can form $G \rtimes N$. Since $H$ is a normal subgroup of $G$, we can also define the subgroup $H \rtimes N$ of $G \rtimes N$. Since $G/H$ is abelian, $H \rtimes N$ is a normal subgroup of $G \rtimes N$.

Let $\theta$ be a smooth character of $S$ and $\text{Irr}(N_H, \theta)$ be the set of irreducible  representations of $N_H$ whose restriction to $S$ is $\theta$-isotypic. We assume that $N$ normalizes the character $\theta$, and every element in $\text{Irr}(N_H, \theta)$ is $1$-dimensional. 
Let $\sigma$ be a smooth of finite-length semisimple representation of $H \rtimes N$. We assume that for $s\in S$, $s^{-1}\rtimes s$ acts on $\sigma$ by $\theta(s)$.
We also assume that
\begin{enumerate}
\item \[\End_H(\sigma) = \bigoplus_{n \in N_{H}/S} \mathbb{C} \cdot \sigma(n^{-1} \rtimes n);\]
\item For each $g\in G$, the representation $^{g \rtimes 1}\!(\sigma\restriction_{H})$ is isomorphic to $\sigma\restriction_{H}$ if $g\in H\cdot N$, and $^{g \rtimes 1}\!(\sigma\restriction_{H})$ and $\sigma\restriction_{H}$ have no common irreducible constituents otherwise;
\item Every irreducible constituent of $\sigma\restriction_{H}$ has the same dimension.
\end{enumerate}
\begin{remark}
\label{rmk1}
We define the representation $\sigma'$ of $H \times N_H$ by $\sigma'(h,n)=\sigma(hn^{-1} \rtimes n)$. Since $S$ is of finite index in $N$, so is $N_{H}$, and $\sigma'$ is semisimple. Then \cite[Lemma~B.1]{2019arXiv191203274K} and the first condition above imply that for every $\widetilde{\theta} \in \text{Irr}(N_H, \theta)$, there exists exactly one irreducible constituent $\tau$ of $\sigma\restriction_{H}$ such that the multiplicity of $\tau \boxtimes \widetilde{\theta}$ in $\sigma'$ is non-zero. Moreover, in this case, the multiplicity of $\tau \boxtimes \widetilde{\theta}$ in $\sigma'$ is $1$. Hence, the third condition above implies that every $\widetilde{\theta}\in \text{Irr}(N_H, \theta)$ has the same the multiplicity in ${\sigma'}\restriction_{N_H}$. 
\end{remark}
Let $\text{Ind}_{H \rtimes N}^{G \rtimes N} \sigma$ be the induced representation on the space
\[\{f \colon G \rtimes N \to \sigma \mid f(xy) = \sigma(x) f(y) \ (x \in H \rtimes N, y \in G \rtimes N)\}.\]
We define the representation $I_{\sigma}$ of $G \times N$ by $I(g,n)= \left(\text{Ind}_{H \rtimes N}^{G \rtimes N} \sigma\right)(gn^{-1} \rtimes n)$.
Since $H$ is of finite index in $G$, $\text{Ind}_{H \rtimes N}^{G \rtimes N} \sigma$ and $I_{\sigma}$ are semisimple.
For $f \in \text{Ind}_{H \rtimes N}^{G \rtimes N} \sigma, g \in G, n \in N, s \in S$, we obtain
\begin{align*}
\left((\text{Ind}_{H \rtimes N}^{G \rtimes N} \sigma)(s^{-1} \rtimes s)f\right)(g \rtimes n)
&=
f\left((g \rtimes n)(s^{-1} \rtimes s)\right)\\
&=
f\left(gns^{-1}n^{-1} \rtimes ns\right)\\
&=
f\left((ns^{-1}n^{-1} \rtimes nsn^{-1})(g \rtimes n)\right)\\
&=
\sigma(ns^{-1}n^{-1} \rtimes nsn^{-1})f(g \rtimes n)\\
&=
\theta(nsn^{-1})f(g \rtimes n)\\
&=
\theta(s)f(g \rtimes n).
\end{align*}
Therefore, $(1, s)$ acts on $I_{\sigma}$ by $\theta$ for $s \in S$.

Let $\text{Irr}(N, \theta)$ be the set of irreducible representations of $N$ whose restriction to $S$ is $\theta$-isotypic, and $[I_{\sigma}\restriction_{G}]$ be the set of irreducible constituents of the $G$-representation $I_{\sigma}\restriction_{G}$.

The representation $I_{\sigma}$ is decomposed as
\[I_{\sigma}=\bigoplus_{\pi \in [I_{\sigma}\restriction_{G}], \rho \in \text{Irr}(N, \theta)} (\pi \boxtimes \rho)^{\oplus m_{\pi, \rho}},\]
where $m_{\pi, \rho}$ is the multiplicity of $\pi \boxtimes \rho$ in $I_{\sigma}$.
\begin{lemma}
\label{keylem}
\begin{enumerate}
\item We have $m_{\pi, \rho} \in \{0,1\}$, and for any $\rho \in \Irr(N, \theta)$, there exists exactly one $\pi \in [I_{\sigma}\restriction_{G}]$ such that $m_{\pi, \rho}=1$. So the condition $m_{\pi, \rho} = 1$ defines a correspondence
\[\Irr(N, \theta) \longleftrightarrow [I_{\sigma}\restriction_{G}].\]
\item For $\rho \in \Irr(N, \theta)$, we write $\pi_{\rho}\in [I_{\sigma}\restriction_{G}]$ for the unique $G$-representation with $m_{\pi_{\rho}, \rho}=1$. Then,
\[\frac{\dim(\pi_{\rho})}{\dim(\rho)}=
\frac{\abs{G/H}\cdot \dim(\sigma)}{\abs{N/S}}.\]
\end{enumerate}
\end{lemma}
\begin{proof}
The first claim follows from \cite[Lemma~B.1]{2019arXiv191203274K} and \cite[Proposition~B.3]{2019arXiv191203274K}. 

We prove the second claim.
Let $\rho \in \text{Irr}(N, \theta)$. Because of the first claim, $\dim(\pi_{\rho})$ is equal to the multiplicity of $\rho$ in $(I_{\sigma})\restriction_{N}$, which is equal to the dimension of
\[\text{Hom}_{N} (\rho, (I_{\sigma})\restriction_{N}).\]
We write $N'=\{n^{-1} \rtimes n \mid n\in N\} \subset G \rtimes N$. By the isomorphism $N' \simeq N$ defined by $n^{-1}\rtimes n \mapsto n$, we regard $\rho$ as a representation of $N'$. Then, 
\begin{align*}
\text{Hom}_{N} (\rho, (I_{\sigma})\restriction_{N}) &= \text{Hom}_{N'} (\rho, (\text{Ind}_{H \rtimes N}^{G \rtimes N} \sigma)\restriction_{N'})\\
&=\text{Hom}_{N'} (\rho, \bigoplus_{g \in N' \backslash (G \rtimes N) / (H \rtimes N)} \text{Ind}_{N' \cap ^g(H \rtimes N)} ^{N'} {^g\!{\sigma}}).
\end{align*}
Let $C$ be a complete system of a representatives for $G/HN$.
For $g\in G, h\in H, n_1, n_2 \in N$, we obtain
\[({n_1}^{-1}\rtimes n_1)(g \rtimes 1)(h \rtimes n_2)=gh{n_1}^{-1} \rtimes {n_1}{n_2}.\]
This calculation implies that the set 
\[\{g\rtimes 1\mid g\in C\}\]
is a complete system of a representatives for $N' \backslash (G \rtimes N )/ (H \rtimes N)$. 

Therefore, we obtain
\[\text{Hom}_{N'} (\rho, \bigoplus_{g \in N' \backslash (G \rtimes N) / (H \rtimes N)} \text{Ind}_{N' \cap ^g(H \rtimes N)} ^{N'} {^g\!{\sigma}})
= \text{Hom}_{N'} (\rho, \, \bigoplus_{g\in C} \text{Ind}_{N' \cap ^{g \rtimes 1}\!(H \rtimes N)} ^{N'} {^{g \rtimes 1}\!{\sigma}}).
\]

For $g \in C, h\in H, n\in N$, we obtain
\[(g \rtimes 1)(h \rtimes n)(g^{-1}\rtimes 1)=ghng^{-1}n^{-1} \rtimes n,\]
which is an element of $N'$ if and only if $ghng^{-1}n^{-1}=n^{-1}$, i.\,e., $h=n^{-1}$.
Moreover, in the case $h=n^{-1}$, we obtain
\[(g \rtimes 1)(h \rtimes n)(g^{-1}\rtimes 1)=(g \rtimes 1)(n^{-1} \rtimes n)(g^{-1}\rtimes 1) = n^{-1} \rtimes n.\]
Let $(N_H)'$ be the set
\[\{{n}^{-1} \rtimes n \mid n\in N_H\}.\]
Then, the calculations above imply that for $g \in C$, 
\[N' \cap ^{g \rtimes 1}\!(H \rtimes N) = (N_H)',\]
and the conjugate action of $g \rtimes 1$ on $(N_H)'$ is trivial. 

Hence, we deduce that
\[\Hom_{N'} (\rho, \, \bigoplus_{g\in C} \text{Ind}_{N' \cap ^{g \rtimes 1}\!(H \rtimes N)} ^{N'} {^{g \rtimes 1}\!{\sigma}})
= \Hom_{N'} (\rho, \, \bigoplus_{g\in C} \text{Ind}_{(N_H)'} ^{N'} \sigma)\] 
By the isomorphism $N' \simeq N, n^{-1}\rtimes n \mapsto n$, $(N_H)'$ maps to $N_H$. Therefore, by regarding $\sigma$ as a representation of $N_H$ by the isomorphism $N_H \simeq (N_H)' \subset H \rtimes N$, which is equal to ${\sigma'}\restriction_{N_H}$ in Remark \ref{rmk1}, we obtain
\begin{align*}
\text{Hom}_{N'} (\rho, \, \bigoplus_{g\in C} \text{Ind}_{(N_H)'} ^{N'} {\sigma'}\restriction_{N_H}) 
&= \text{Hom}_{N} (\rho, \, \bigoplus_{g\in C} \text{Ind}_{N_H} ^{N} {\sigma'}\restriction_{N_H})\\
&= \bigoplus_{g\in C} \text{Hom}_{N} (\rho, \text{Ind}_{N_H}^{N} {\sigma'}\restriction_{N_H})\\
& \simeq \bigoplus_{g\in C} \text{Hom}_{N_H} (\rho, {\sigma'}\restriction_{N_H}).
\end{align*}
According to Remark \ref{rmk1}, there exists an integer $m$ such that ${\sigma'}\restriction_{N_H}$ is decomposed as
\[{\sigma'}\restriction_{N_H} 
= \left(
\bigoplus_{\widetilde{\theta}\in \text{Irr}(N_H, \theta)}
\widetilde{\theta} 
\right)^{\oplus m}.
\]
Since every element in $\text{Irr}(N_H, \theta)$ is $1$-dimensional, we obtain
\[m = \frac{\dim(\sigma)}{\abs{\text{Irr}(N_H, \theta)}}=\frac{\dim(\sigma)}{\abs{N_H/S}}.\]
Therefore, if we write the multiplicity of $\widetilde{\theta} \in \text{Irr}(N_H, \theta)$ in $\rho\restriction_{N_H}$ by $m_{\rho}(\widetilde{\theta})$,
the dimension of $\text{Hom}_{N}(\rho, (I_{\sigma})\restriction_{N})$ is equal to
\begin{align*}
\abs{C}\cdot m \cdot \sum_{\widetilde{\theta}\in \Irr(N_H, \theta)} m_{\rho}(\widetilde{\theta})&=
\abs{G/HN}\cdot m \cdot \sum_{\widetilde{\theta}\in \Irr(N_H, \theta)} m_{\rho}(\widetilde{\theta})\\
&=\abs{G/HN}\cdot m \cdot \dim(\rho)\\
&= \frac{\abs{G/HN}\cdot \dim(\sigma) \cdot \dim(\rho)}{\abs{N_H/S}}\\
&= \dim(\rho)\cdot \frac{\abs{G/H}\cdot \dim(\sigma)}{\abs{N/S}}.
\end{align*}
This completes the proof.
\end{proof}
\section{Formal degree of non-singular supercuspidal representations}
\label{secfdns}
In this section, we calculate the formal degree of non-singular supercuspidal representations.

Let $(S, \theta)$ be a tame $k_F$-non-singular elliptic pair and $\rho \in \text{Irr}(N(S, G)(F)_{\theta}, \theta)$.
We calculate the formal degree of $\pi_{(S, \theta, \rho)}^{\epsilon}$ defined in Section~\ref{secnonsing}. Since $\pi_{(S, \theta, \rho)}^{\epsilon}$ is constructed by twisted Yu's construction from the generic cuspidal $G$-datum
\[(\overrightarrow{G}, y, \overrightarrow{r}, \kappa_{(S, \phi_{-1}, \rho_{-1})}^{\epsilon}, \overrightarrow{\phi}),\]
Proposition~\ref{thmA} and Remark~\ref{twisted} imply that the formal degree $d(\pi_{(S, \theta, \rho)}^{\epsilon})$ of $\pi_{(S, \theta, \rho)}^{\epsilon}$ is equal to
\[\frac{\dim(\kappa_{(S, \phi_{-1}, \rho_{-1})}^{\epsilon})}{[G^{\text{a}, 0}(F)_{[y]} : G^{\text{a}, 0}(F)_{y, 0+}]} \text{exp}_{q} \left(\frac{1}{2} \dim(G^{\text{a}}) + \frac{1}{2} \dim(\mathsf{G^{\text{a}, 0}})_{[y]}^\circ) + \frac{1}{2}\sum_{i=0}^{d-1} r_i \left(\abs{R(G^{i+1}, S) - R(G^{i}, S)}\right)\right).\]
Next, we calculate $\dim(\kappa_{(S, \phi_{-1}, \rho_{-1})}^{\epsilon})$ by using the result in Section~\ref{seclem}.
We apply Lemma~\ref{keylem} for,
\begin{align*}
G & = \mathsf{G^{0}}_{[y]} (k_F) = G^0(F)_{[y]}/ G^0(F)_{y, 0+}, \\
H &= \mathsf{G^{0}}'_{[y]}(k_F) = S(F)G^0(F)_{y, 0}/G^0(F)_{y, 0+}, \\
S & = \mathsf{S}(k_F) = S(F)/S(F)_{0+},\\
N &= N(\mathsf{S}, \mathsf{G^{0}}_{[y]})(k_F)_{\phi_{-1}}, \\
N_H &= N \cap H = N(\mathsf{S}, \mathsf{G^{0}}'_{[y]})(k_F)_{\phi_{-1}}, \\
\theta &= (\phi_{-1})^{-1}.
\end{align*}
According to \cite[Corollary~2.2.2]{2019arXiv191203274K} and \cite[Proposition~2.3.3]{2019arXiv191203274K}, every element in 
\[\text{Irr}(N(\mathsf{S}, \mathsf{G^{0}}'_{[y]})(k_F)_{\phi_{-1}}, (\phi_{-1})^{-1})\]
is $1$-dimensional. 

In the proof of \cite[Theorem~2.7.7]{2019arXiv191203274K}, Kaletha defined an action $C_{\mathsf{U}}^{\mathsf{G^{0}}'_{[y]}, \epsilon}$ of $N(\mathsf{S}, \mathsf{G^{0}}_{[y]})(k_F)_{\phi_{-1}}$ on $H^{d(\mathsf{U})}_{c} (Y_{\mathsf{U}}^{\mathsf{G^{0}}'_{[y]}}, \overline{\mathbb{Q}_{\ell}})_{\phi_{-1}}$.
The action $C_{\mathsf{U}}^{\mathsf{G^{0}}'_{[y]}, \epsilon}$ dose not commute with the action of $\kappa_{(\mathsf{S}, \phi_{-1})}^{\mathsf{G^0}'_{[y]}}$, but instead translates it as
\[C_{\mathsf{U}}^{\mathsf{G^{0}}'_{[y]}, \epsilon}(n)\circ \kappa_{(\mathsf{S}, \phi_{-1})}^{\mathsf{G^0}'_{[y]}}(h)=\kappa_{(\mathsf{S}, \phi_{-1})}^{\mathsf{G^0}'_{[y]}}(nhn^{-1})\circ C_{\mathsf{U}}^{\mathsf{G^{0}}'_{[y]}, \epsilon}(n)\]
for $n\in N(\mathsf{S}, \mathsf{G^{0}}_{[y]})(k_F)_{\phi_{-1}}$ and $h\in \mathsf{G^{0}}'_{[y]}(k_F)$.
Therefore, by using $C_{\mathsf{U}}^{\mathsf{G^{0}}'_{[y]}, \epsilon}(n)$, we can extend the representation $\kappa_{(\mathsf{S}, \phi_{-1})}^{\mathsf{G^0}'_{[y]}}$ to a representation $\sigma$ of
\[H \rtimes N = \mathsf{G^{0}}'_{[y]}(k_F) \rtimes N(\mathsf{S}, \mathsf{G^{0}}_{[y]})(k_F)_{\phi_{-1}}.\]
Kaletha also defined an action
$C_{\mathsf{U}}^{\mathsf{G^{0}}_{[y]}, \epsilon}$
of $N(\mathsf{S}, \mathsf{G^{0}}_{[y]})(k_F)_{\phi_{-1}}$ on $H^{d(\mathsf{U})}_{c} (Y_{\mathsf{U}}^{\mathsf{G^{0}}_{[y]}}, \overline{\mathbb{Q}_{\ell}})_{\phi_{-1}}$, which satisfies
\[C_{\mathsf{U}}^{\mathsf{G^{0}}_{[y]}, \epsilon}(n)=\kappa_{(\mathsf{S}, \phi_{-1})}^{\mathsf{G^0}_{[y]}}(n) \circ R_{\mathsf{U}}^{\mathsf{G^{0}}_{[y]}, \epsilon}(n)\]
for $n\in N(\mathsf{S}, \mathsf{G^{0}}_{[y]})(k_F)_{\phi_{-1}} \subset \mathsf{G^{0}}_{[y]} (k_F)$.
Using this action, he extended the representation $\kappa_{(\mathsf{S}, \phi_{-1})}^{\mathsf{G^0}_{[y]}}$ to a representation $\Sigma$ of 
\[G \rtimes N = \mathsf{G^{0}}_{[y]} (k_F) \rtimes N(\mathsf{S}, \mathsf{G^{0}}_{[y]})(k_F)_{\phi_{-1}}\]
and proved that
\[\Sigma\simeq \text{Ind}_{\mathsf{G^{0}}'_{[y]}(k_F) \rtimes N(\mathsf{S}, \mathsf{G^{0}}_{[y]})(k_F)_{\phi_{-1}}}^{\mathsf{G^{0}}_{[y]} (k_F)  \rtimes N(\mathsf{S}, \mathsf{G^{0}}_{[y]})(k_F)_{\phi_{-1}}} \sigma.\]
Therefore, we obtain that the representation $\kappa_{(\mathsf{S}, \phi_{-1})}^{\mathsf{G^0}_{[y]}, \epsilon}$ is isomorphic to the representation 
\[(g, n) \mapsto \left(\text{Ind}_{\mathsf{G^{0}}'_{[y]}(k_F) \rtimes N(\mathsf{S}, \mathsf{G^{0}}_{[y]})(k_F)_{\phi_{-1}}}^{\mathsf{G^{0}}_{[y]} (k_F)  \rtimes N(\mathsf{S}, \mathsf{G^{0}}_{[y]})(k_F)_{\phi_{-1}}} \sigma\right)(gn^{-1} \rtimes n)\]
of $\mathsf{G^{0}}_{[y]} (k_F) \times N(\mathsf{S}, \mathsf{G^{0}}_{[y]})(k_F)_{\phi_{-1}}$.

Since the restriction of $R_{\mathsf{U}}^{\mathsf{G^{0}}_{[y]}, \epsilon}$ to $\mathsf{S}(k_F)$ is $(\phi_{-1})^{-1}$-isotypic, the action of $s^{-1} \rtimes s$ by $\Sigma$ is the multiple of $(\phi_{-1})^{-1}(s)$ for $s \in \mathsf{S}(k_F)$.
In particular, the action of $s^{-1}\rtimes s$ by $\sigma$ is also the multiple of $(\phi_{-1})^{-1}(s)$ for $s\in \mathsf{S}(k_F)$.

We will verify the assumptions in section~\ref{seclem};
\begin{enumerate}
\item \[\End_H(\sigma) = \bigoplus_{n \in N_{H}/S} \mathbb{C} \cdot \sigma(n^{-1} \rtimes n);\]
\item For each $g\in G$, the representation $^{g \rtimes 1}\!(\sigma\restriction_{H})$ is isomorphic to $\sigma\restriction_{H}$ if $g\in H\cdot N$, and $^{g \rtimes 1}\!(\sigma\restriction_{H})$ and $\sigma\restriction_{H}$ have no common irreducible constituents otherwise;
\item Every irreducible constituent of $\sigma\restriction_{H}$ has the same dimension.
\end{enumerate}
The first two conditions are verified in the proof of \cite[Theorem~2.7.7]{2019arXiv191203274K}. We consider the third point.
We note that $\sigma\restriction_{H}$ is isomorphic to $\kappa_{(\mathsf{S}, \phi_{-1})}^{\mathsf{G^0}'_{[y]}}$, which is obtained by endowing $\kappa_{(\mathsf{S}^\circ, \theta^\circ)}^{\mathsf{G^0}_{[y]}^\circ}$ with a structure of $\mathsf{G^0}'_{[y]}(k_F)$-representation (see Remark~\ref{rmkg'}).

Let $N(\mathsf{S}, \mathsf{G^0}_{[y]}^\circ)(k_F)_{\phi_{-1}}$ and $N(\mathsf{S}^\circ, \mathsf{G^0}_{[y]}^\circ)(k_F)_{\phi_{-1}^\circ}$ be the stabilizer of the pair $(\mathsf{S}, \phi_{-1})$ and $(\mathsf{S}^\circ, \phi_{-1}^\circ)$ in $\mathsf{G^0}_{[y]}^\circ(k_F)$ by the conjugate action respectively. For $*=N(\mathsf{S}, \mathsf{G^0}_{[y]}^\circ)(k_F)_{\phi_{-1}}$ or $N(\mathsf{S}^\circ, \mathsf{G^0}_{[y]}^\circ)(k_F)_{\phi_{-1}^\circ}$, we define $\text{Irr}(*, \phi_{-1}^{\circ})$ be the set of irreducible representations of $*$ whose restriction to $\mathsf{S}^\circ(k_F)$ is $\phi_{-1}^{\circ}$-isotypic.
According to \cite[Corollary~2.2.2]{2019arXiv191203274K} and \cite[Proposition~2.3.3]{2019arXiv191203274K}, every element in $\text{Irr}(N(\mathsf{S}, \mathsf{G^0}_{[y]}^\circ)(k_F)_{\phi_{-1}}, \phi_{-1}^{\circ})$ and $\text{Irr}(N(\mathsf{S}^\circ, \mathsf{G^0}_{[y]}^\circ)(k_F)_{\phi_{-1}^\circ}, \phi_{-1}^\circ$) are $1$-dimensional.
It is shown in the proof of \cite[Theorem~2.7.7]{2019arXiv191203274K} that the irreducible constituents of $\kappa_{(\mathsf{S}^\circ, \phi_{-1}^\circ)}^{\mathsf{G^0}_{[y]}^\circ}$ are indexed by the set $\text{Irr}(N(\mathsf{S}^\circ, \mathsf{G^0}_{[y]}^\circ)(k_F)_{\phi_{-1}^\circ}, \phi_{-1}^{\circ})$, and the irreducible constituents of $\kappa_{(\mathsf{S}, \phi_{-1})}^{\mathsf{G^0}'_{[y]}}$ are indexed by the set $\text{Irr}(N(\mathsf{S}, \mathsf{G^0}_{[y]}^\circ)(k_F)_{\phi_{-1}}, \phi_{-1}^{\circ})$. 

We now explain the relationship between the irreducible constituents of $\kappa_{(\mathsf{S}^\circ, \phi_{-1}^\circ)}^{\mathsf{G^0}_{[y]}^\circ}$ and the irreducible constituents of $\kappa_{(\mathsf{S}, \phi_{-1})}^{\mathsf{G^0}'_{[y]}}$.
Let $\rho \in \text{Irr}(N(\mathsf{S}, \mathsf{G^0}_{[y]}^\circ)(k_F)_{\phi_{-1}}, \phi_{-1}^{\circ})$ and $E_{\rho}$ be the set of the representation $\rho^{\circ}$ of $N(\mathsf{S}^\circ, \mathsf{G^0}_{[y]}^\circ)(k_F)_{\phi_{-1}^\circ}$ whose restriction to $N(\mathsf{S}, \mathsf{G^0}_{[y]}^\circ)(k_F)_{\phi_{-1}}$ is equal to $\rho$.
\begin{lemma}
The cardinality of the set $E_{\rho}$ is equal to 
\[\abs{N(\mathsf{S}^\circ, \mathsf{G^0}_{[y]}^\circ)(k_F)_{\phi_{-1}^\circ}/N(\mathsf{S}, \mathsf{G^0}_{[y]}^\circ)(k_F)_{\phi_{-1}}}.\]
In particular, it is independent of $\rho$.
\end{lemma}
\begin{proof}
Since $E_{\rho}$ is an $N(\mathsf{S}^\circ, \mathsf{G^0}_{[y]}^\circ)(k_F)_{\phi_{-1}^\circ}/N(\mathsf{S}, \mathsf{G^0}_{[y]}^\circ)(k_F)_{\phi_{-1}}$-torsor, it is enough to show that $E_{\rho}$ is not empty.
According to \cite[Proposition~2.3.3]{2019arXiv191203274K}, there exists an extension $\rho^{\circ}$ of $\phi_{-1}$ to $N(\mathsf{S}^\circ, \mathsf{G^0}_{[y]}^\circ)(k_F)_{\phi_{-1}^\circ}$. Then, 
\[\rho \cdot (\rho^{\circ}\restriction_{N(\mathsf{S}, \mathsf{G^0}_{[y]}^\circ)(k_F)_{\phi_{-1}}})^{-1}\]
is a character of $N(\mathsf{S}, \mathsf{G^0}_{[y]}^\circ)(k_F)_{\phi_{-1}}/\mathsf{S}(k_F)$.
Since $N(\mathsf{S}^\circ, \mathsf{G^0}_{[y]}^\circ)(k_F)_{\phi_{-1}^\circ}/\mathsf{S}(k_F)$ is abelian \cite[Corollary~2.2.2]{2019arXiv191203274K}, we can extend 
\[\rho \cdot (\rho^{\circ}\restriction_{N(\mathsf{S}, \mathsf{G^0}_{[y]}^\circ)(k_F)_{\phi_{-1}}})^{-1}\]
to a character $\chi$ of $N(\mathsf{S}^\circ, \mathsf{G^0}_{[y]}^\circ)(k_F)_{\phi_{-1}^\circ}/\mathsf{S}(k_F)$.
Then, $\rho^{\circ} \cdot \chi$ is an element of $E_{\rho}$.
\end{proof}
In the proof of \cite[Theorem~2.7.7]{2019arXiv191203274K}, Kaletha proved that the irreducible constituent of $\kappa_{(\mathsf{S}, \phi_{-1})}^{\mathsf{G^0}'_{[y]}}$ corresponding to $\rho$ is the direct sum of the irreducible constituents of $\kappa_{(\mathsf{S}^\circ, \phi_{-1}^\circ)}^{\mathsf{G^0}_{[y]}^\circ}$ corresponding to $\rho^{\circ} \in E_{\rho}$.

On the other hand, According to \cite[2.3]{2019arXiv191203274K}, the conjugate action of the $k_F$-point of the adjoint group of $\mathsf{G^0}_{[y]}^\circ$ on the set of irreducible constituents of $\kappa_{(\mathsf{S}^\circ, \theta^\circ)}^{\mathsf{G^0}_{[y]}^\circ}$ is transitive. In particular, every irreducible constituent of $\kappa_{(\mathsf{S}^\circ, \theta^\circ)}^{\mathsf{G^0}_{[y]}^\circ}$ has the same dimension. Let $m$ denote the dimension of irreducible constituents of $\kappa_{(\mathsf{S}^\circ, \theta^\circ)}^{\mathsf{G^0}_{[y]}^\circ}$.
Then, the argument above implies that the dimensions of irreducible constituents of $\kappa_{(\mathsf{S}, \phi_{-1})}^{\mathsf{G^0}'_{[y]}}$ are equal to
\[m \cdot \abs{N(\mathsf{S}^\circ, \mathsf{G^0}_{[y]}^\circ)(k_F)_{\phi_{-1}^\circ}/N(\mathsf{S}, \mathsf{G^0}_{[y]}^\circ)(k_F)_{\phi_{-1}}}.\]
In particular, the third assumption is verified.

Now, Lemma~\ref{keylem} implies the following theorem, which is a generalization of \cite[Corollary~52]{2021arXiv210100658S}.
\begin{theorem}
\label{keythm}
The formal degree $d(\pi_{(S, \theta, \rho)}^{\epsilon})$ of $\pi_{(S, \theta, \rho)}^{\epsilon}$ is equal to
\[\frac{\dim(\rho)\cdot\exp_{q} \left(
\frac{1}{2} \dim(G^{\Aa}) +
\frac{1}{2}\rank((\mathsf{G^{\Aa, 0}})_{[y]}^\circ) +
\frac{1}{2}\sum_{i=0}^{d-1} r_i \left(\abs{R(G^{i+1}, S) - R(G^{i}, S)}\right)
\right)}{\abs{N(S, G)(F)_{\theta}/S(F)}\abs{S^{\Aa}(F)/S^{\Aa}(F)_{0+}}}
.
\] 
\end{theorem}
\begin{proof}
Recall that the formal degree $d(\pi_{(S, \theta, \rho)}^{\epsilon})$ of $\pi_{(S, \theta, \rho)}^{\epsilon}$ is equal to
\[\frac{\dim(\kappa_{(S, \phi_{-1}, \rho_{-1})}^{\epsilon})}{[G^{\text{a}, 0}(F)_{[y]} : G^{\text{a}, 0}(F)_{y, 0+}]} \exp_{q} \left(\frac{1}{2} \dim(G^{\text{a}}) + \frac{1}{2} \dim((\mathsf{G^{\text{a}, 0}})_{[y]}^\circ) + \frac{1}{2}\sum_{i=0}^{d-1} r_i \left(\abs{R(G^{i+1}, S) - R(G^{i}, S)}\right)\right).\]  
According to Lemma~\ref{keylem}, the dimension of $\kappa_{(S, \phi_{-1}, \rho_{-1})}^{\epsilon}$ is equal to
\[\frac{\dim(\rho)}{\abs{N(\mathsf{S}, \mathsf{G}^{0}_{[y]})(k_F)_{\phi_{-1}}/\mathsf{S}(k_F)}} \abs{\mathsf{G^{0}}_{[y]}(k_F)/\mathsf{G^{0}}'_{[y]}(k_F)} \dim(\kappa_{(\mathsf{S}^\circ, \theta^\circ)}^{\mathsf{G^0}_{[y]}^\circ}).\]
According to \cite[Lemma~3.2.2]{2019arXiv191203274K} and \cite[Lemma~3.4.5]{2019arXiv191203274K}, we obtain the isomorphism
\begin{align*}
N(\mathsf{S}, \mathsf{G}^{0}_{[y]})(k_F)_{\phi_{-1}}/\mathsf{S}(k_F)
& \simeq N(S, G^0)(F)_{\phi_{-1}}/S(F)\\
& \simeq N(S, G)(F)_{\theta}/S(F).
\end{align*}
Moreover, \cite[Corollary~6.4.3]{10013123614} and \cite[Corollary~7.2]{MR393266} imply that the dimension $\dim(\kappa_{(\mathsf{S}^\circ, \theta^\circ)}^{\mathsf{G^0}_{[y]}^\circ})$ of $\kappa_{(\mathsf{S}^\circ, \theta^\circ)}^{\mathsf{G^0}_{[y]}^\circ}$ is equal to
\[\frac{[\mathsf{G^0}_{[y]}^\circ(k_F): \mathsf{S}^\circ(k_F)]}{\text{exp}_q \left(\frac{1}{2} \left(\dim(\mathsf{G^0}_{[y]}^\circ) - \dim(\mathsf{S}^\circ)\right)\right)}.\]
Therefore, we obtain
\begin{align*}
&\abs{\mathsf{G^{0}}_{[y]}(k_F)/\mathsf{G^{0}}'_{[y]}(k_F)} \dim(\kappa_{(\mathsf{S}^\circ, \theta^\circ)}^{\mathsf{G^0}_{[y]}^\circ})\\ 
= &[G^0(F)_{[y]} : S(F)G^0(F)_{y,0}] \cdot \frac{[\mathsf{G^0}_{[y]}^\circ(k_F): \mathsf{S}^\circ(k_F)]}{\text{exp}_q \left(\frac{1}{2} \left(\dim(\mathsf{G^0}_{[y]}^\circ) - \dim(\mathsf{S}^\circ)\right)\right)}\\
= &[G^0(F)_{[y]} : S(F)G^0(F)_{y,0}] \cdot \frac{[G^0(F)_{y, 0}: S(F)_{0}G^0(F)_{y, 0+}]}{\text{exp}_q \left(\frac{1}{2} \left(\dim(\mathsf{G^0}_{[y]}^\circ) - \dim(\mathsf{S}^\circ)\right)\right)}\\
= &[G^{\text{a},0}(F)_{[y]} : S^{\text{a}}(F)G^{\text{a}, 0}(F)_{y, 0}] \cdot \frac{[G^{\text{a}, 0}(F)_{y, 0}: S^{\text{a}}(F)_{0}G^{\text{a}, 0}(F)_{y, 0+}]}{\text{exp}_q \left(\frac{1}{2} \left(\dim((\mathsf{G^{\text{a}, 0}})_{[y]}^\circ) - \rank((\mathsf{G^{\text{a}, 0}})_{[y]}^\circ)
\right)\right)}.\\
\end{align*}
Then, the claim follows from a computation.
\end{proof}
\section{Comparison}
\label{seccomp}
In this section, we prove that Kaletha's construction of the local Langlands correspondence \cite{2019arXiv191203274K} for torally wild $L$-parameters satisfies Conjecture~\ref{fdc}.

Let $\varphi$ be a torally wild $L$-parameter which corresponds to the torally wild $L$-packet datum $(S, \widehat{j}, \chi_0, \theta)$.
Conjugating $\varphi$ in $\widehat{G}$ if necessary, we may assume that the image $\widehat{T}$ of $\widehat{j}\colon \widehat{S} \to \widehat{G}$ is $\Gamma_F$-stable.
As explained in \cite[4.2]{2019arXiv191203274K}, we may assume that the $\chi$-data for $R(G, S)$ obtained from $\theta$ as in \cite[3.5]{2019arXiv191203274K} is equal to the $\chi$-data for $R(G, S)$ obtained from $\chi_0$ as in Section~\ref{seclpara}. 
We need not to concern ourselves with the precise way to obtain the $\chi$-data $(\chi_{\alpha})_{\alpha}$ for $R(G, S)$ from $\theta$, but we only need to know that each $\chi_{\alpha}$ is tamely ramified of finite order. 
In particular, the hypothesis of \cite[Corollary~67]{2021arXiv210100658S} is satisfied, hence \cite[Corollary~67]{2021arXiv210100658S} also holds in our situation.

Let $j\colon S \to G$ be an admissible embedding defined over $F$.
By using $j$, we regard $A$ as a subgroup of $S$, and let $S^{\text{a}}$ denote the torus $S/A$.

Assume that $\pi \in [\pi_{(jS, \theta_j)}]$ corresponds to $\rho \in \text{Irr}(\pi_0(S_{\varphi}^{+}), \eta)$ via the bijection~\eqref{bijc7}.
We write $\pi=\pi_{(jS, \theta_j, \rho_j)}^{\epsilon}$ for some $\rho_{j} \in \text{Irr}(N(jS, G)(F)_{\theta_j}, \theta_j)$. By the construction of the bijection~\eqref{bijc7}, we obtain 
\[\dim(\rho_j)= \frac{\dim(\rho)}{\abs{\pi_0(S_{\varphi}^{+})/\pi_0(S_{\varphi}^{+})_{\eta}}}.\]
On the other hand, Theorem~\ref{keythm} implies that the formal degree of $\pi_{(jS, \theta_j, \rho_j)}^{\epsilon}$ is equal to
\[\frac{\dim(\rho_j)\cdot\exp_{q} \left(
\frac{1}{2} \dim(G^{\text{a}}) +
\frac{1}{2}\text{rank}((\mathsf{G^{\text{a}, 0}})_{[y]}^\circ) +
\frac{1}{2}\sum_{i=0}^{d-1} r_i \left(\abs{R(G^{i+1}, jS) - R(G^{i}, jS)}\right)
\right)}{\abs{N(jS, G)(F)_{\theta_j}/jS(F)}\abs{S^{\text{a}}(F)/S^{\text{a}}(F)_{0+}}}.
\]
Therefore, we conclude that the formal degree of $\pi_{(jS, \theta_j, \rho_j)^{\epsilon}}$ is equal to
\[\frac{\dim(\rho)\cdot\exp_{q} \left(
\frac{1}{2} \dim(G^{\text{a}}) +
\frac{1}{2}\text{rank}((\mathsf{G^{\text{a}, 0}})_{[y]}^\circ) +
\frac{1}{2}\sum_{i=0}^{d-1} r_i \left(\abs{R(G^{i+1}, jS) - R(G^{i}, jS)}\right)
\right)}
{\abs{\pi_0(S_{\varphi}^{+})/\pi_0(S_{\varphi}^{+})_{\eta}}\abs{N(jS, G)(F)_{\theta_j}/jS(F)}\abs{S^{\text{a}}(F)/S^{\text{a}}(F)_{0+}}}.
\]

Next, we calculate the $\gamma$-factor.
We apply the calculation of $\abs{\gamma(0, \pi, \text{Ad}, \psi)}$ in \cite[Section~4]{2021arXiv210100658S} for our situation.
Recall that $\varphi={^{L}\!j}\circ \varphi_{S}$, where $\varphi_{S}\colon W_F \to {^{L}S}$ is the $L$-parameter for the character $\theta$, and $^{L}\!j\colon {^{L}S} \to {^{L}G}$ is an extension of $\widehat{j}$ defined by using the $\chi$-data for $R(S, G)$ obtained from $\chi_0$.
Combining with the adjoint representation $\text{Ad}$ of ${^{L}G}$ on $\widehat{\mathfrak{g}}/\widehat{\mathfrak{z}}^{\Gamma_{F}}$, we define the representation $\text{Ad}\circ \varphi$ of $W_F$. Here, $\widehat{\mathfrak{g}}$ and $\widehat{\mathfrak{z}}$ are the Lie algebras of $\widehat{G}$ and its center respectively. We calculate the absolute value $\abs{\gamma(0,\text{Ad}\circ \varphi)}$ of the $\gamma$-factor.

As explained in \cite[4.3]{2021arXiv210100658S}, the representation $\text{Ad}\circ \varphi$ decomposes as a direct sum
\[\widehat{\mathfrak{g}}/\widehat{\mathfrak{z}}^{\Gamma_{F}}=: V = V_{\text{toral}} \oplus V_{\text{root}},\]
where $V_{\text{toral}}=\widehat{\mathfrak{t}}/\widehat{\mathfrak{z}}^{\Gamma_{F}}$, and
\[V_{\text{root}}= \bigoplus_{\alpha \in R(G, S)} \widehat{\mathfrak{g}}_{\alpha}.\]
Here, $\widehat{\mathfrak{t}}$ denotes the Lie algebra of $\widehat{T}$, and $\widehat{\mathfrak{g}}_{\alpha}$ is the usual $\check{\alpha}$-eigenspace for $\widehat{T}$, where $\check{\alpha}$ is interpreted as a root of $(\widehat{G}, \widehat{T})$ via the map $\widehat{j}\colon \widehat{S} \to \widehat{T}$.
Now, we explain the calculation of $\abs{\gamma(0, V_{\text{toral}})}$ and $\abs{\gamma(0, V_{\text{root}})}$ \cite[Section~4]{2021arXiv210100658S}.

First, we explain the calculation of $\abs{\gamma(0, V_{\text{toral}})}$.
As explained in \cite[4.4]{2021arXiv210100658S}, 
\[V_{\text{toral}} \simeq \mathbb{C}\otimes X^{*}(S^{\text{a}})\] 
as representations of $W_F$.
Let $M=X^{*}(S^{\text{a}})^{I_F}$ and write $M^{\vee}$ for the dual lattice of $M$.
Then, $\abs{\gamma(0, V_{\text{toral}})}$ is calculated as follows.
\begin{lemma}[{\cite[Lemma~69]{2021arXiv210100658S}}]
\label{toral}
The absolute value $\abs{\gamma(0, V_{\toral})}$ is equal to
\[\exp_{q}\left(\frac{1}{2}\left(
\dim(S^{\Aa})+ \rank(M)
\right)\right)
\frac{\abs{M_{\Frob}}}
{\abs{(\overline{k_F}^{\times}\otimes M^{\vee})^{\Frob}}}
.\]
\end{lemma}
\begin{remark}
In \cite[Section~4]{2021arXiv210100658S}, Schwein assumes that the $L$-parameter $\varphi$ is regular in the sense of \cite[Definition~5.2.3]{MR4013740}.
However, the argument in \cite[Lemma~69]{2021arXiv210100658S} dose not require the regularity condition. Hence, \cite[Lemma~69]{2021arXiv210100658S} also holds in our situation.
\end{remark}

Next, we explain the calculation of $\abs{\gamma(0, V_{\text{root}})}$ \cite[4.5]{2021arXiv210100658S}.
Let $\underline{R}(G, S)$ be the set of $\Gamma_F$-orbits in $R(G, S)$.
The root summand $V_{\text{root}}$ decomposes as a direct sum
\[V_{\text{root}}=\bigoplus_{\underline{\alpha} \in \underline{R}(G, S)} V_{\underline{\alpha}},\]
where
\[V_{\underline{\alpha}}=\bigoplus_{\alpha\in \underline{\alpha}} V_{\alpha}.\]
Moreover, by choosing $\alpha\in \underline{\alpha}$, we can identify $V_{\underline{\alpha}}$ with the induced representations $\text{Ind}_{W_{\alpha}}^{W_F} \psi_{\alpha}$, where $W_{\alpha}$ denotes the stabilizer of $\alpha$ in $W_F$, and $\psi_{\alpha}$ denotes the character of $W_{\alpha}$ defined in \cite[4.5]{2021arXiv210100658S}. 
Schwein calculated the depths of characters $\psi_{\alpha}$ \cite[Corollary~71, Lemma~74]{2021arXiv210100658S} to obtain the value of $\abs{\gamma(0, V_{\text{root}})}$.
We will follow his arguments in our situation.

For $\alpha \in R(G, S)$, let $\Gamma_{\alpha}$ be the stabilizer of $\alpha$ in $\Gamma_{F}$ and $F_{\alpha}$ be the corresponding fixed subfield of $F^{\text{sep}}$.
We consider the character 
\[\theta \circ N_{F_{\alpha}/F} \circ \check{\alpha}\colon F_{\alpha}^{\times} \to \mathbb{C}^{\times}, \]
where $N_{F_{\alpha}/F}$ denotes the norm map $S(F_{\alpha}) \to S(F)$.
Recall that the depth of $\theta \circ N_{F_{\alpha}/F} \circ \check{\alpha}$ is defined as
\[\inf\{r\in \mathbb{R}_{> 0}\mid \left(\theta \circ N_{F_{\alpha}/F} \circ \check{\alpha}\right)\left((F_{\alpha})_{r}^{\times}\right)=1\}.\]
\begin{lemma}[{\cite[Corollary~71]{2021arXiv210100658S}}]
\label{pos}
If the depth of $\theta \circ N_{F_{\alpha}/F} \circ \check{\alpha}$ is positive, then it is equal to the depth of $\psi_{\alpha}$. 
\end{lemma}
\begin{proof}
We recall that \cite[Corollary~67]{2021arXiv210100658S} also holds in our situation.
Then, the claim follows from the same argument as \cite[Corollary~71]{2021arXiv210100658S}.
\end{proof}

The case where the depth of $\theta \circ N_{F_{\alpha}/F} \circ \check{\alpha}$ is equal to $0$ is as follows.
\begin{lemma}[{\cite[Lemma~74]{2021arXiv210100658S}}]
\label{zero}
If the depth of $\theta \circ N_{F_{\alpha}/F} \circ \check{\alpha}$ is equal to $0$, then the depth of $\psi_{\alpha}$ is equal to $0$. Moreover, $\psi_{\alpha}$ is not an unramifed character, i.\,e., $\psi_{\alpha}$ is not trivial on the inertia subgroup of $W_{\alpha}$.
\end{lemma}
\begin{proof}
According to \cite[Lemma~72]{2021arXiv210100658S} and \cite[Remark~73]{2021arXiv210100658S}, the claim follows when $jS$ is maximally unramified in $G$ in the sense of \cite[Definition~3.4.2]{MR4013740}.

Moreover, according to \cite[Proposition~6.9]{2019arXiv190705173K}, we obtain the decomposition ${^{L}\!j}={^{L}\!j_{G^0, G}}\circ {^{L}\!j_{S, G^0}}$, where $G^0$ is the twisted Levi subgroup of $G$ with maximal torus $jS$ and satisfying $R(G^0, S)=R_{0+}$, $^{L}j_{S, G^0}\colon {^{L}S }\to {^{L}G^{0}}$ is the extension of $\widehat{S} \to \widehat{G^0}$ obtained from the unique minimally ramified $\chi$-data for $R(G^0, S)$, and $^{L}\!j_{G^0, G}\colon \widehat{G^0} \to \widehat{G}$ is the extension of $\widehat{G^0} \to \widehat{G}$ obtained from the $\chi$-data $\chi_0$ for $R(G, S^0)$ \cite[6.1]{2019arXiv190705173K}. 
Since the depth of $\theta \circ N_{F_{\alpha}/F} \circ \check{\alpha}$ is equal to $0$, we obtain from the definition of $R_{0+}$ that $\alpha\in R(G^0, S)$. From the construction of $^{L}\!j_{G^0, G}$, we obtain that the embedding $\widehat{\mathfrak{g}}^0 \to \widehat{\mathfrak{g}}$ is $^{L}\!j_{G^0, G}$-equivariant, where $\widehat{\mathfrak{g}}^0$ denotes the Lie algebra of $G^0$.
In this way, we reduce the case $G=G^0$, where $jS$ is maximally unramified in $G$.
\end{proof}
Now, we calculate $\abs{\gamma(0, V_{\text{root}})}$.
Recall that the $F$-non-singular pair $(jS, \theta_j)$ determines a sequence of twisted Levi subgroups
\[\overrightarrow{G}=\left(jS=G^{-1}\subset G^0\subsetneq \ldots \subsetneq G^d=G\right)\]
and a sequence of real numbers $\overrightarrow{r}=\left(0=r_{-1}, r_0, \ldots , r_d\right)$ \cite[3.6]{MR4013740}.
\begin{lemma}[{\cite[Lemma~76]{2021arXiv210100658S}}]
\label{root}
The absolute value $\abs{\gamma(0, V_{\Root})}$ is equal to 
\[\exp_{q}\left(
\frac{1}{2}\abs{R(G, jS)} + \frac{1}{2}\sum_{i=0}^{d} r_i(\abs{R(G^{i+1}, jS)}-\abs{R(G^i, jS)})
\right).\]
\end{lemma}
\begin{proof}
The claim follows from Lemma~\ref{pos} and \ref{zero} as \cite[Lemma~76]{2021arXiv210100658S}.
\end{proof}
Combining Lemma~\ref{toral} and Lemma~\ref{root}, we obtain the following result.
\begin{proposition}
The absolute value $\abs{\gamma(0, \pi, \Ad, \psi)}$ is equal to
\[\frac{\abs{M_{\Frob}}}
{\abs{(\overline{k_F}^{\times}\otimes M^{\vee})^{\Frob}}}\exp_q\left(
\frac{1}{2}\dim(G^{\Aa}) + \frac{1}{2}\rank(M) + \frac{1}{2}\sum_{i=0}^{d} r_i(\abs{R(G^{i+1}, jS)}-\abs{R(G^i, jS)})
\right).\]
\end{proposition}
According to \cite[Lemma~78]{2021arXiv210100658S}, the rank of $M=X^{*}(S^{\text{a}})^{I_F}$ is equal to the rank of $(\mathsf{G^{\text{a}, 0}})_{[y]}^\circ$.
Moreover, according to \cite[Lemma~5.13, Lemma~5.17, Lemma~5.18]{MR3402796}, we see
\[\frac{\abs{M_{\text{Frob}}}}
{\abs{(\overline{k_F}^{\times}\otimes M^{\vee})^{\text{Frob}}}}
=
\frac{\abs{\pi_0([\widehat{\overline{S}}]^{\natural})}}{\abs{S^{\text{a}}(F)/S^{\text{a}}(F)_{0+}}},
\]
where $[\widehat{\overline{S}}]^{\natural}$ denotes the preimage of $\widehat{S}^{\Gamma_F}$ via the map $\widehat{S/A} \to \widehat{S}$, and $\pi_0([\widehat{\overline{S}}]^{\natural})$ denotes the group of connected components of $[\widehat{\overline{S}}]^{\natural}$ (see \cite[Section~5]{2021arXiv210100658S}).
Hence, we obtain
\[\abs{\gamma(0, \pi, \text{Ad}, \psi)}= \frac{\abs{\pi_0([\widehat{\overline{S}}]^{\natural})}\cdot\exp_{q} \left(
\frac{1}{2} \dim(G^{\text{a}}) +
\frac{1}{2}\text{rank}((\mathsf{G^{\text{a}, 0}})_{[y]}^\circ) +
\frac{1}{2}\sum_{i=0}^{d-1} r_i \left(\abs{R(G^{i+1}, jS) - R(G^{i}, jS)}\right)
\right)}
{\abs{S^{\text{a}}(F)/S^{\text{a}}(F)_{0+}}}.\]
Therefore, to prove Conjecture~\ref{fdc}, it is enough to show the following lemma.
\begin{lemma}
\[
\abs{\pi_0(S_{\varphi}^{+})/\pi_0(S_{\varphi}^{+})_{\eta}}\abs{N(jS, G)(F)_{\theta_j}/jS(F)}
=
\frac{\abs{\pi_0(S_{\varphi}^{\natural})}}{\abs{\pi_0([\widehat{\overline{S}}]^{\natural})}}.
\]
\end{lemma}
\begin{proof}
According to \cite[Proposition~4.3.2]{2019arXiv191203274K}, $j$ induces an exact sequence
\[1\to \widehat{S}^{\Gamma_F} \to S_{\varphi} \to \Omega(S, G)(F)_{\theta} \to 1.\]
From this exact sequence, we obtain an exact sequence
\[1\to \pi_0([\widehat{\overline{S}}]^{\natural}) \to \pi_0(S_{\varphi}^{\natural}) \to \Omega(S, G)(F)_{\theta} \to 1.\]
Therefore, the right hand side of the lemma is equal to $\abs{\Omega(S, G)(F)_{\theta}}$.

On the other hand, \cite[Lemma~3.4.10]{MR4013740} and \cite[Lemma~E.1]{2019arXiv191203274K} imply that $j$ gives an isomorphism
\[\Omega(S, G)(F)_{\theta, \eta} \longrightarrow N(jS, G)(F)_{\theta_j}/jS(F).\]

Moreover, \cite[Corollary~4.3.4]{2019arXiv191203274K} and its pull back along the inclusion $\Omega(S, G)(F)_{\theta, \eta} \to \Omega(S, G)(F)_{\theta}$ give the two exact sequences
\begin{align*}
&1 \to \pi_0([\widehat{\overline{S}}]^{+}) \to \pi_0(S_{\varphi}^{+}) \to \Omega(S, G)(F)_{\theta} \to 1,\\
&1 \to \pi_0([\widehat{\overline{S}}]^{+}) \to \pi_0(S_{\varphi}^{+})_{\eta} \to \Omega(S, G)(F)_{\theta, \eta} \to 1.
\end{align*}
Hence, we obtain that
\begin{align*}
(\text{LHS}) &= \abs{\pi_0(S_{\varphi}^{+})/\pi_0(S_{\varphi}^{+})_{\eta}}\abs{N(jS, G)(F)_{\theta_j}/jS(F)}\\
&=\abs{\pi_0(S_{\varphi}^{+})/\pi_0([\widehat{\overline{S}}]^{+})}{\abs{\pi_0(S_{\varphi}^{+})_{\eta}/\pi_0([\widehat{\overline{S}}]^{+})}}^{-1} \abs{\Omega(S, G)(F)_{\theta, \eta}}\\
&= \abs{\Omega(S, G)(F)_{\theta}}{\abs{\Omega(S, G)(F)_{\theta, \eta}}}^{-1} \abs{\Omega(S, G)(F)_{\theta, \eta}}\\
&= \abs{\Omega(S, G)(F)_{\theta}}\\
&=(\text{RHS}).
\end{align*}
\end{proof}
\providecommand{\bysame}{\leavevmode\hbox to3em{\hrulefill}\thinspace}
\providecommand{\MR}{\relax\ifhmode\unskip\space\fi MR }
\providecommand{\MRhref}[2]{%
  \href{http://www.ams.org/mathscinet-getitem?mr=#1}{#2}
}
\providecommand{\href}[2]{#2}

\end{document}